\documentclass[12pt]{amsart}
\usepackage{amscd}
\usepackage{mathrsfs}
\usepackage{amsmath}
\usepackage[inline]{enumitem}
\numberwithin{equation}{section}

\newtheorem{thm}{Theorem}[section]
\newtheorem*{theorem*}{Theorem}
\newtheorem{theorem}{Theorem}
\newtheorem{lemma}[thm]{Lemma}
\newtheorem{corollary}[thm]{Corollary}
\newtheorem{proposition}[thm]{Proposition}

\theoremstyle{definition}

\newtheorem{remark}[thm]{Remark}

\newtheorem{definition}[thm]{Definition}
\newtheorem{claim}[thm]{Claim}

\newtheorem{assumptions}[thm]{Assumptions}

\newtheorem{defn-thm}[thm]{Definition-Theorem}

\usepackage{amsmath}
\usepackage{amsthm}
\usepackage{commath}
\usepackage{amssymb}
\usepackage{hyperref}
\usepackage{tikz-cd}
\usepackage{mathtools}   

\newcommand{\Teich}{\mathcal{T}}
\newcommand{\Sym}{\mathrm{Sym}}
\newcommand{\Curve}{\mathcal{V}}
\newcommand{\Mod}{\mathrm{Mod}}

\title{Constant energy families of harmonic maps}
\author{Ognjen To\v{s}i\'c}

\begin{document}
\maketitle

\begin{abstract}
    For a negatively curved manifold $M$ and a continuous map $\psi:\Sigma\to M$ from a closed surface $\Sigma$, we study complex submanifolds of Teichm\"uller space $\mathcal{S}\subset\Teich(\Sigma)$ such that the harmonic maps $\{h_X:X\to M\text{ for }X\in\mathcal{S}\}$ in the homotopy class of $\psi$ all have equal energy. When $M$ has negative Hermitian sectional curvature, we show that for any such $\mathcal{S}$, there exists a closed Riemann surface $Y$, such that any $h_X$ for $X\in\mathcal{S}$ factors as a holomorphic map $\phi_X:X\to Y$ followed by a fixed harmonic map $h:Y\to M$. This answers a question posed by both Toledo and Gromov. As a first application, we show a factorization result for harmonic maps from normal projective varieties to $M$. As a second application, we study homomorphisms from finite index subgroups of mapping class groups to $\pi_1(M)$.
\end{abstract}

\section{Introduction}
Let $M$ be a manifold with negative sectional curvatures, and $\psi:\Sigma_g\to M$ be a continuous map from a closed surface of genus $g$, denoted $\Sigma_g$. Then given any marked Riemann surface $f:\Sigma_g\to X$ in Teichm\"uller space $\mathcal{T}_g$, there exists a unique harmonic map $h:X\to M$ in the homotopy class of $\psi\circ f^{-1}$. The energy of the map $h$ then defines a function $\mathrm{E}_\psi:\mathcal{T}_g\to \mathbb{R}$. It was shown by Toledo \cite{Toledo2011} that when $M$ has strictly negative Hermitian sectional curvature, $\mathrm{E}_\psi$ is plurisubharmonic, and that it is strictly plurisubharmonic when the harmonic map $h$ is an immersion.
\par In the converse direction, Schiffer variations provide examples of complex submanifolds of $\Teich_g$ on which $\mathrm{E}_\psi$ is constant. Given a fixed Riemann surface $Y$ marked by $\Sigma_{\tilde{g}}$, and a ramified covering map $p:\Sigma_g\to\Sigma_{\tilde{g}}$ branched over some finite set $B\subset\Sigma_{\tilde{g}}$, by moving the branch set $B$ on $Y$, we obtain a complex submanifold $\mathrm{Schiff}_{Y, p}\subset\mathcal{T}_g$, consisting of marked Riemann surfaces that all admit a holomorphic map to $Y$ in the homotopy class of $p$. It is then easy to see that $\mathrm{E}_{\psi\circ p}$ is constant on $\mathrm{Schiff}_{Y,p}$, for an arbitrary continuous map $\psi:\Sigma_{\tilde{g}}\to M$. 
\par A natural question, posed by both Toledo \cite[Remark 3]{Toledo2011} and Gromov \cite[Remarks (a), \S 4.6, pp. 53]{Gromov2012}, is do all examples of families of homotopic harmonic maps with constant energy arise in this way? Our main result answers in the affirmative.
\begin{theorem}\label{thm:inframain}
    Let $\mathcal{S}\subset\Teich_g$ be a connected complex submanifold of Teichm\"uller space, and let $\psi:\Sigma_g\to M$ be a continuous map into a convex cocompact Riemannian manifold $M$ with negative Hermitian sectional curvature, such that $\psi$ is not homotopic into a curve. Then, if $\mathrm{E}_\psi$ is constant on $\mathcal{S}$, there exist a closed Riemann surface $Y$ marked by $\Sigma_h$, and a continuous map $p:\Sigma_g\to\Sigma_h$, such that any $X\in\mathcal{S}$ admits a holomorphic map to $Y$ homotopic to $p$, and such that $\psi\simeq \theta\circ p$ for some continuous map $\theta:\Sigma_h\to M$. 
\end{theorem}
Note that in the equivariant case, an example of constant energy families of harmonic maps was provided by Deroin and Tholozan \cite[Theorem 5]{Deroin2019}.
\par In the setting of Theorem \ref{thm:inframain}, for any $X\in\Teich_g$, the harmonic map $h:X\to M$ homotopic to $\psi$ factors as $h=h_\theta\circ h_p$, where $h_p:X\to Y$ is the holomorphic map homotopic to $p$, and $h_\theta:Y\to M$ is the harmonic map homotopic to $\theta$. More precisely, Theorem \ref{thm:inframain} is essentially equivalent to a factorization result for holomorphic fibrations with (complex) one-dimensonal fibres. 
\begin{definition}\label{dfn:admissible}
    We call a holomorphic map $p:E\to B$ between connected complex manifolds an \textit{admissible fibration} if there exists a proper analytic subset $A\subsetneq B$ such that $p:E\setminus p^{-1}(A)\to B\setminus A$ is a holomorphic submersion with connected, closed and (complex) one-dimensional fibres, such that $p:E\setminus p^{-1}(A)\to B\setminus A$ is topologically a product.
\end{definition}
\begin{theorem}\label{thm:supermain}
    Let $M$ be a convex cocompact Riemannian manifold with negative Hermitian sectional curvature, and $p:E\to B$ be an admissible fibration. Then for any smooth map $F:E\to M$ transverse to $p$ such that the maps $F_b=F|_{p^{-1}(b)}:p^{-1}(b)\to M$, for $b\in B\setminus A$, are all harmonic of equal energy, we have either 
    \begin{enumerate}
        \item $F$ has image in a closed geodesic in $M$, or 
        \item $F$ factors as $h\circ\phi$, where $\phi:E\to Y$ is a holomorphic map to a closed hyperbolic Riemann surface $Y$ transverse to $p$, and $h:Y\to M$ is a harmonic map.
    \end{enumerate}
\end{theorem} 
\subsection{Factorization of harmonic maps}\label{subsec:intro-ph}
As an application of Theorem \ref{thm:supermain}, we show the following.
\begin{theorem}\label{thm:ph-smooth-projective}
    Let $X$ be a normal projective variety equipped with the structure of a K\"ahler space, and let $M$ be a convex cocompact Riemannian manifold with negative Hermitian sectional curvature. Then for any non-constant continuous map $f:X\to M$, harmonic on the smooth locus of $X$, we have either 
    \begin{enumerate}
        \item $f$ is a pluriharmonic map to a closed geodesic in $M$, or 
        \item $f$ factors as $h\circ\phi$, where $\phi:X\to Y$ is a holomorphic map to a closed hyperbolic Riemann surface $Y$, and $h:Y\to M$ is harmonic.
    \end{enumerate}
\end{theorem}
When $X$ is a compact K\"ahler manifold, and when the universal cover of $M$ is a (real) hyperbolic space $\mathbb{H}^d$, this is the result of Carlson--Toledo \cite[Theorem 7.2 (a)]{Carlson1989}. Since in our case, the variety $X$ is allowed to have mild singularities, we work in the setting of K\"ahler spaces (for precise definitions, the reader can consult \cite{Varouchas1989}).
\par Since we make use of Theorem \ref{thm:supermain}, we require the variety we start with to have many (complex) one-dimensional submanifolds. In the smooth projective case, this is easily obtained through a modification of Bertini's theorem (see \S\ref{subsec:bertini}). Using Hironaka's resolution of singularities \cite{Hironaka1964,Hironaka1964a} this can be extended to arbitrary projective varieties.
\subsection{Mapping class groups} As a second application of Theorem \ref{thm:supermain}, we consider virtual properties of mapping class groups. We denote by $\Mod_{g,n}$ the pure mapping class group of a surface $\Sigma_{g,n}$ of genus $g$ with $n$ punctures. Recall the Birman exact sequence \cite[\S 4.2]{Farb2011}
\begin{align*}
    1\to \pi_1(\Sigma_{g,n})\stackrel{\iota}{\longrightarrow}\Mod_{g,n+1}\stackrel{\mathcal{F}}{\longrightarrow}\Mod_{g,n}\to 1.
\end{align*}
Here $\mathcal{F}$ is the forgetful homomorphism obtained by filling in a puncture, and $\iota$ realizes $\pi_1(\Sigma_{g,n})$ as {the point-pushing subgroup} $\Pi_{g,n}\leq\Mod_{g,n+1}$.
\begin{definition}
    Given a subgroup $\Gamma\leq\Mod_{g,n+1}$ and a group $G$, a homomorphism $\phi:\Gamma\to G$ is called \textit{strongly point-pushing} if $\phi(\Gamma\cap\Pi_{g,n})$ is not trivial or cyclic.
\end{definition}
\begin{theorem}\label{thm:mcg}
    Let $M$ be a convex cocompact Riemannian manifold with negative Hermitian sectional curvature, and let $\Gamma\leq\Mod_{g,n+1}$ be a finite index subgroup with $g\geq 3, n\geq 0$. Then for any strongly point-pushing homomorphism $\phi:\Gamma\to\pi_1(M)$, 
    \begin{enumerate}
        \item\label{item:geometry}there exists a closed hyperbolic Riemann surface $Y$, such that for any marked Riemann surface $X\in\Teich_{g,n}$, its cover $\hat{X}$ that corresponds to the subgroup $\iota^{-1}(\Pi_{g,n}\cap\Gamma)$ admits a non-constant holomorphic map to $Y$, and 
        \item\label{item:topology} there exist a finite index subgroup $\Theta\leq\Gamma$ and a homomorphism $\theta:\Theta\to\pi_1(Y)$ such that $\phi|_{\Theta}$ factors through $\theta$.
    \end{enumerate}
\end{theorem}
Note that when $M$ is a one-holed torus equipped with a complete hyperbolic metric, Theorem \ref{thm:mcg}(\ref{item:geometry}) recovers one direction of a theorem of Markovi\'c \cite[Theorem 1.2, ``if'' direction]{Markovic2022}.
\subsection{Outline and organization} We first show Theorems \ref{thm:inframain} and \ref{thm:supermain} in \S\ref{sec:prelims}--\S\ref{sec:pf-thm-main}. We then show Theorem \ref{thm:ph-smooth-projective} in \S\ref{sec:factorization-results}, and Theorem \ref{thm:mcg} in \S\ref{sec:mcg}.
\subsubsection{Main results} Our point of view in showing Theorems \ref{thm:inframain} and \ref{thm:supermain} is Teichm\"uller theoretic. In \S\ref{sec:prelims} we explain how to derive Theorem \ref{thm:supermain} in general, assuming the case when $B$ is a submanifold of Teichm\"uller space and $E$ is the restriction of the universal curve to this submanifold. This special case is essentially Theorem \ref{thm:inframain}, and \S\ref{sec:ph}--\S\ref{sec:pf-thm-main} are devoted to its proof.
\par We now give an outline of the proof of Theorem \ref{thm:inframain}. Let $\mathcal{S}\subset\Teich_g$ be a submanifold, and $\psi:\Sigma_g\to M$ a continuous map with $\mathrm{E}_\psi$ constant on $\mathcal{S}$. Let $\pi:\Curve_g\to\Teich_g$ be the universal curve, and write $\Curve_\mathcal{S}=\pi^{-1}(\mathcal{S})$. We also write $F_\psi:\Curve_g\to M$ for the map that is harmonic and homotopic to $\psi$ on the fibres of $\pi$. Our aim is to show that $F_\psi$ factors through a holomorphic map to a closed Riemann surface $Y$.
\par The proof relies on three ingredients. 
\\\par \textbf{Ingredient 1. Foliation $\mathcal{F}$. } Using that $\mathrm{E}_\psi$ is constant, we show that $F_\psi$ is pluriharmonic and $\dim_\mathbb{C} DF_\psi(T^{1,0}\Curve_\mathcal{S})\leq 1$. This is a local argument, carried out in Proposition \ref{prop:energy-constant-iff-ph} in \S\ref{sec:ph}. We define a holomorphic foliation $\mathcal{F}$ on $\Curve_\mathcal{S}\setminus\Delta$, where $\dim_\mathbb{C}\Delta\leq\dim_\mathbb{C}\mathcal{S}-1$, and where the leaves of $\mathcal{F}$ are tangent to the $\ker\partial F_\psi$. In particular, $F_\psi$ is constant on $\mathcal{F}$. Properties of $\mathcal{F}$ can be found in Corollary \ref{cor:hol-structure}.
    \\\par \textbf{Ingredient 2. Schiffer varieties $\Lambda_{Y,p}$. } For a Riemann surface $Y$ and a continuous map $p:\Sigma_g\to Y$, define 
    \begin{align*}
        \Lambda_{Y,p}=\{X\in\Teich_g:\text{there exists a holomorphic map }X\to Y\text{ homotopic to }p\}
    \end{align*}
    In \S\ref{sec:schiffer}, we show that $\Lambda_{Y,p}$ is a complex closed submanifold of $\Teich_g$, and that there is a holomorphic map $\Curve_{\Lambda_{Y,p}}\to Y$ homotopic to $p$. \par Note that this space is stratified by $\mathrm{Schiff}_{Y,\sigma}$, where $\sigma$ ranges over branched covers $\Sigma_g\to Y$ homotopic to $p$. Rather than describing the combinatorics of how the spaces $\mathrm{Schiff}_{Y,\sigma}$ fit together, we choose here to show directly that $\Lambda_{Y,p}$ is a closed complex submanifold using the non-abelian Hodge correspondence and the previous work of the author \cite{Tosic2023}. Along the way, we also show the analogous result for indiscrete representations $\pi_1(\Sigma_g)\to\mathrm{PSL}(2,\mathbb{R})$. 
    \\\par \textbf{Ingredient 3. Equivalence relations $\sim$. } Fix a general point $t\in\mathcal{S}$ and let $X_t=\pi^{-1}(t)$ be the Riemann surface that corresponds to $t$, and consider the harmonic map $F_\psi|_{X_t}:X_t\to M$. By the work of Sagman \cite{Sagman2021}, there exists a maximal relation $\sim$ on $X_t$ through which $F_\psi|_{X_t}$ factors holomorphically (Definition \ref{dfn:sim}). Specifically, for $x,y \in X_t$ with $(\partial F_\psi|_{X_t})_x,(\partial F_\psi|_{X_t})_y\neq 0$, let $x\sim y$ if there are neighbourhoods $U_x$ (resp. $U_y$) of $x$ (resp. $y$), and a bihlomorphism $\phi:U_x\to U_y$ such that $F_\psi\circ\phi=F_\psi|_{U_x}$. It is shown in \cite{Sagman2021} that $X_t\to Y_t:=X_t/\sim$ is a finite (ramified) cover of closed Riemann surfaces (see \S\ref{subsec:holomorphic-equivalence-relation} for more precise references). Moreover, by definition we have $F_\psi|_{X_t}=h_t\circ q_t$, where $h_t:Y_t\to M$ is harmonic.
\\ \par We claim that $Y_t$ are isomorphic for $t$ ranging over some open subset of $\mathcal{S}$. Fix an arbitrary $t\in\mathcal{S}$. Since $F_\psi$ factors through the local leaf spaces of $\mathcal{F}$, it follows that $F_\psi|_{X_t}$ factors through a holomorphic map $X_t\to Y_s$ followed by $h_s:Y_s\to M$. Thus by definition of $\sim$ on $X_t$, it follows that $Y_t$ admits a non-constant holomorphic map to $Y_s$. It easily follows from this that the surfaces $\{Y_t:t\in\mathcal{S}\}$ are all isomorphic over some open subset of $\mathcal{S}$.
\par By a simple application of the Baire category theorem, it follows that $q_t:X_t\to Y_t$ are all homotopic and $Y_t$ all isomorphic for $t\in\mathcal{U}\subset\mathcal{S}$, over some open subset $\mathcal{U}$. Let $Y=Y_t$ and $q:\Sigma_g\to Y$ be homotopic to $q_t$, for some fixed $t\in\mathcal{U}$. Thus $\mathcal{U}\subset\Lambda_{Y,q}$. Since $\Lambda_{Y,q}$ is closed, it follows that $\mathcal{S}\subset\Lambda_{Y,q}$, and the proof of Theorem \ref{thm:inframain} is completed. 
\subsubsection{Factorization result} We show Theorem \ref{thm:ph-smooth-projective} in \S\ref{sec:factorization-results}. The singular version is easily obtained from the case when $X$ is smooth using Hironaka's theorems on resolution of singularities \cite{Hironaka1964,Hironaka1964a} combined with results on singularity removal for harmonic maps \cite{Meier1986}, so we will assume that $X$ is smooth in this outline.
\par Suppose that $X\subseteq\mathbb{P}^N$ is $d$-dimensional. By considering all possible intersections of $X$ with an $(N-d+1)$-dimensional subspace of $\mathbb{P}^N$, we obtain a diagram of holomorphic maps between smooth projective varieties 
\begin{equation*}
\begin{tikzcd}
    E \arrow[r, "p"]\arrow[d, "\phi"] & B \\ 
    X 
\end{tikzcd}    
\end{equation*}
where the general fibre of $p$ is a closed connected Riemann surface, by Bertini's theorem. Here $B$ parameterizes $(N-d+1)$-dimensional subspaces of $\mathbb{P}^N$. 
\par From the Siu--Sampson theorem \cite{Sampson1986,Siu1980} (the reader can also consult \cite[Chapter 6]{Amoros2014} or \cite{Loustau2020}) and our local analysis in \S\ref{sec:ph} (see Proposition \ref{prop:energy-constant-iff-ph}), it follows that $(f\circ\phi)|_{p^{-1}(b)}$ has energy independent of $b\in B$. 
\par Note that $p:E\to B$ is not topologically a product, so we are still unable to apply Theorem \ref{thm:supermain}. We therefore take the universal cover $\tilde{B}$ of $B$, which gives an immersion $\tilde{B}\looparrowright\Lambda_{Y,q}$, for some closed Riemann surface $Y$ and a continuous map $q:\Sigma_g\to Y$. In particular, we get a holomorphic map $\Curve_{\tilde{B}}\to Y$. An argument using the De Franchis--Severi theorem implies that this map descends to a holomorphic map $\theta:\hat{E}\to Y$, where $\hat{E}$ is the pullback of $E$ via a finite cover $\hat{B}\to B$. 
\par Using standard properties of harmonic maps \cite{aronszajn1956unique}, we can show that $\theta$ is constant on the connected components of the fibres of $\hat{E}\to E\stackrel{\phi}{\to}X$. Therefore we obtain a map $\hat{\theta}:X\to\Sym^k Y$ for an appropriate integer $k$, by mapping $x\in X$ to the $\theta$-image of the lifts to $\hat{E}$ of the connected components of $\phi^{-1}(x)\subset E$. We then show that the image of $\hat{\theta}$ has complex dimension one, and that $f$ factors through $\hat{\theta}$. Therefore the map $f$ factors through the holomorphic map from $X$ to a Riemann surface obtained by taking the normalization of the image of $\hat{\theta}$. 
\subsubsection{Mapping class groups} We show Theorem \ref{thm:mcg} in \S\ref{sec:mcg}, following the ideas in the joint paper of Markovi\'c and the author \cite{Markovic2024}.
\par Let $\Gamma\leq\Mod_{g,n+1}$ be a finite index subgroup, and $\phi:\Gamma\to\pi_1(M)$ a strongly point-pushing homomorphism. This defines a homomorphism $\phi\circ\iota:\pi_1(\Sigma_{g,n})\to\pi_1(M)$, with image not contained in a cyclic subgroup of $\pi_1(M)$, that is invariant under $\Gamma$. 
\par By a result of Bridson \cite{Bridson2010}, it follows that $\ker\phi$ contains all Dehn multitwists in $\Gamma$. In particular, it contains any elements of $\Pi_{g,n}$ that correspond to simple closed curves, including small loops around the punctures. Let $\psi:\pi_1(\Sigma_g)\to\pi_1(M)$ be the extension of $\phi\circ\iota$ over these punctures.
\par We consider the energy functional $\mathrm{E}_\psi:\Teich_g\to \mathbb{R}$. This functional is plurisubharmonic, and invariant under the image $\Gamma'$ of $\Gamma$ under the forgetful homomoprhism $\Mod_{g,n}\to\Mod_g$. Thus $\mathrm{E}_\psi$ descends to a plurisubharmonic map on a finite cover $\mathcal{M}$ of the moduli space of genus $g$. An analysis based on the same result of Bridson \cite{Bridson2010} implies that $\mathrm{E}_\psi$ is bounded on $\mathcal{M}$. Since $\mathcal{M}$ is quasiprojective, standard results in complex analysis imply that $\mathrm{E}_\psi$ is constant. Theorem \ref{thm:mcg} then follows immediately by Theorem \ref{thm:supermain}.
\section{Reduction to submanifolds of Teichm\"uller space}\label{sec:prelims}
In this section, we set the stage for the proof of Theorem \ref{thm:inframain}. We first introduce some notation in the setup of Theorem \ref{thm:inframain} in \S\ref{subsubsec:notation}, and then we show Theorem \ref{thm:supermain} assuming Theorem \ref{thm:inframain} in \S\ref{subsec:top-product}.
\subsection{Notation}\label{subsubsec:notation}
Let $\pi:\mathcal{V}_g\to\Teich_g$ be the universal curve over Teichm\"uller space, that is a complex fibration such that the fibre of $\pi$ over $S\in\Teich_g$ is biholomorphic to $S$. For a complex submanifold $\mathcal{S}$ of $\Teich_g$, we denote $\Curve_\mathcal{S}=\pi^{-1}(\mathcal{S})$. Given a homotopy class $\psi:\Sigma_g\to M$ for some convex cocompact manifold $M$ with negative sectional curvatures, we let
\begin{align*}
    F_\psi:\Curve_g\to M
\end{align*}
be the map that is harmonic and homotopic to $\psi$ on every fibre of $\pi:\Curve_g\to\Teich_g$.
\begin{remark}
    It follows from the classical work of Eells--Lemaire \cite{Eells1981} that the harmonic map $X\to M$ in the homotopy class of $\psi$ depends smoothly on $X\in\mathcal{T}_g$ (the reader can consult the work of Slegers \cite{Slegers2021} or previous work of the author \cite{Tosic2023} for the analogous result when the target is a non-compact symmetric space). In particular, the map $F_\psi$ is smooth.
\end{remark}
In the entirety of this paper, we will assume the following. 
\begin{assumptions}\label{ass:mfd}
    Let $M$ be a convex cocompact Riemannian manifolds that has negative Hermitian sectional curvature.
\end{assumptions}

\begin{definition}\label{dfn:schiffer}
    Given a continuous map $p:\Sigma_h\to \Sigma_{{g}}$, we define the \textit{Schiffer variety} $\Lambda_{Y,p}$ associated to $Y\in\Teich_{g}$ and $p$ to be the set of marked Riemann surfaces $X\in\Teich_h$ that admit a holomorphic map to $Y$ in the homotopy class of $p$.
\end{definition}
We will show in Proposition \ref{prop:schiffer} that Schiffer varieties are complex submanifolds of Teichm\"uller space $\Teich_h$, closed as subsets of $\Teich_h$. Moreover, we will show that the corresponding map $F_p:\Curve_{\Lambda_{Y,p}}\to Y$ is holomorphic.
\begin{remark}
    Given a covering map $\sigma:\Sigma_{h,m}\to\Sigma_{g,n}$, we get an isometric embedding of Teichm\"uller spaces $\sigma^*:\Teich_{g,n}\to\Teich_{h,m}$, obtained by lifting complex structures from $\Sigma_{g,n}$ to $\Sigma_{h,m}$ via $p$. Given a Riemann surface $Y\in\Teich_g$, let $\Delta_Y\subset\Teich_{g,n}$ be the preimage of $Y$ under the forgetful morphism $\Teich_{g,n}\to\Teich_g$. Then, if $\sigma$ is homotopic to $p$ after filling in the punctures, we see that $\sigma^*(\Delta_Y)$ projects to $\mathrm{Schiff}_{Y,\sigma}\subset\Lambda_{Y,p}$ under the forgetful morphism $\Teich_{h,m}\to\Teich_h$. These are the subsets of $\Teich_h$ that we referred to as Schiffer variations in the introduction. It is easy to see that $\{\mathrm{Schiff}_{Y,\sigma}:\sigma\simeq p\}$ stratify $\Lambda_{Y, p}$. 
\end{remark}
In \S\ref{subsec:top-product} we show Theorem \ref{thm:supermain} assuming Theorem \ref{thm:inframain}.
\subsection{Reducing Theorem \ref{thm:supermain} to Theorem \ref{thm:inframain}}\label{subsec:top-product}
Let $M, p, E, B, F, A$ be as in the statement of Theorem \ref{thm:supermain}, and assume Theorem \ref{thm:inframain}.
\par By the universal property of Teichm\"uller space, there exists a holomorphic map $\iota:B\setminus A\to \Teich_g$, such that $\iota^*\Curve_g=E\setminus p^{-1}(A)$. In fact, we abuse notation slightly to denote by $\iota$ both the natural map $B\setminus A\to\Teich_g$ and the natural map that makes the following diagram commute 
\[\begin{tikzcd}
    E\setminus p^{-1}(A) \arrow[r, "\iota"]\arrow[d, "p"] & \Curve_g\arrow[d, "\pi"] \\
    B\setminus A \arrow[r, "\iota"] &  \Teich_g
\end{tikzcd}\]
\begin{lemma}\label{lm:simply-connected-base-factor}
    Let $B$ be a connected complex manifold, $\iota:B\to\Teich_g$ a holomorphic map, and $M$ be a negatively curved manifold. Then for any continuous map $F:\iota^*\Curve_g\to M$ transverse to $p:=\iota^*\pi:\iota^*\Curve_g\to B$, such that $F_b=F|_{p^{-1}(b)}:p^{-1}(b)\to M$ is a harmonic map for any $b\in B$, we have  
    \begin{align*}
        F=F_\psi\circ\iota
    \end{align*}
    for some continuous map $\psi:\Sigma_g\to M$.
\end{lemma}
\begin{proof}
    Since the fibres of $p:\iota^*\Curve_g\to B$ are identified with the fibres of $\pi:\Curve_g\to\Teich_g$ via $\iota$, they each have a natural marking. Let $\psi:\Sigma_g\to M$ be the continuous map, well-defined up to homotopy, such that $F|_{p^{-1}(b)}$ is in the homotopy class of $\psi$, for any $b\in B$.
    \par Now let $U\subset B$ be an open subset where $F$ has constant rank. Then by the constant rank theorem, we can locally write ${\iota}$ over a possibly smaller open set $V\subseteq U$ as 
    \begin{align*}
        V\stackrel{s}{\longrightarrow} \mathcal{S}\subset \Teich_g,
    \end{align*}
    where $s$ is a submersion and $\mathcal{S}$ is a complex submanifold of $\Teich_g$. Then $F:s^*\Curve_\mathcal{S}\to M$ is harmonic on the fibres $\pi^{-1}(t)$ for $t\in\mathcal{S}$. Since $M$ is negatively curved, the Jacobi operator of any harmonic map from a closed surface into $M$ is injective. Therefore $F$ descends to a map $\bar{F}:\Curve_\mathcal{S}\to M$. Note that $\bar{F}|_{\pi^{-1}(t)}\simeq\psi$ for any $t\in\mathcal{S}$, so $\bar{F}=F_\psi$. Thus $F=F_\psi\circ\iota$ on $V$.
    \par Since this equality holds on the open dense subset where $\iota$ has locally constant rank, by continuity we have $F=F_\psi\circ\iota$ everywhere on $B$.
\end{proof}
Applying Lemma \ref{lm:simply-connected-base-factor}, we get $F=F_\psi\circ\iota$ over $E\setminus p^{-1}(A)$. If $\psi$ is null-homotopic, then $F_\psi$ is constant on the fibres of $\pi$. Then $F$ is constant on the fibres of $p$, which is a contradiction since $F$ is transverse to $p$. If $\psi$ is homotopic into a closed curve $\gamma$, then the image of $F_\psi$ lies in the closed geodesic that corresponds to $\gamma$. Then the image of $F$ also lies in this closed geodesic, and Theorem \ref{thm:supermain} is shown. Assume therefore that $\psi$ is not homotopic into a curve.  
\par Note that the energy of $F$ over $p^{-1}(b)$ is equal to $\mathrm{E}_\psi\circ\iota(b)$. Therefore $\mathrm{E}_\psi$ is constant on $\iota(B\setminus A)$. By Theorem \ref{thm:inframain}, the set $\iota(B\setminus A)$ is contained in a Schiffer variety $\Lambda\subset\Teich_g$, and $F_\psi=h\circ\phi$, where $\phi:\Curve_\Lambda\to Y$ is a holomorphic map to a closed hyperbolic Riemann surface $Y$. By the Schwarz lemma, the map $\phi\circ\iota:E\setminus p^{-1}(A)\to Y$ is meromorphic in the sense of Andreotti \cite{Andreotti1960}, and hence extends to a map $\bar{\phi}:E\to Y$ by \cite[Theorem 4]{Andreotti1960}. Therefore $F=h\circ\bar{\phi}$ on $E\setminus p^{-1}(A)$, and hence by continuity on $E$, and Theorem \ref{thm:supermain} is shown.
\section{Pluriharmonicity of the combined map}\label{sec:ph}
Our main result in this section is the following characterization of constant energy $\mathrm{E}_\psi$ complex submanifolds of Teichm\"uller space in terms of the map $F_\psi$, that we believe to be of independent interest.
\begin{proposition}\label{prop:energy-constant-iff-ph}
    Let $\psi:\Sigma_g\to M$ be a continuous map into a Riemannian manifold $M$. Suppose that the pair $(M,\psi)$ satisfies either of the following two conditions 
    \begin{enumerate}
        \item\label{item:negative} $M$ has negative Hermitian sectional curvature and is convex cocompact, and $\psi$ is not homotopic into a closed curve in $M$, or 
        \item $M$ is a locally symmetric space of non-positive curvature, and the representation $\psi_*:\pi_1(\Sigma_g)\to \mathrm{Isom}(\tilde{M})$ has trivial centralizer.
    \end{enumerate}
    Let $\mathcal{S}\subset\Teich_g$ be a complex submanifold. Then $\mathrm{E}_\psi$ is constant on $\mathcal{S}$ if and only if the map $F_\psi$ is pluriharmonic and $R^M(\partial F_\psi|_{\ker\pi_*}\wedge \partial F_\psi)=0$ on $\mathcal{V}_\mathcal{S}$.
\end{proposition}
\begin{remark}
    Suppose we are in the setting of Proposition \ref{prop:energy-constant-iff-ph}(\ref{item:negative}). Then observe that the conclusions of the proposition can be restated as 
    \begin{align*}
        \bar{\partial}\partial F_\psi=0\text{ and }DF_\psi(T^{1,0}\Curve_\mathcal{S})\leq 1,
    \end{align*}
    which are the same as the conclusions of the Siu--Sampson rigidity theorem \cite{Siu1980,Sampson1986}. In particular, the Siu--Sampson theorem can be shown (for targets with strictly negative Hermitian sectional curvature) using Proposition \ref{prop:energy-constant-iff-ph} using the ideas of Gromov \cite[\S 4.6]{Gromov2012} for complex manifolds that are closed Riemann surface bundles over closed complex manifolds. 
\end{remark}
We will mainly use Proposition \ref{prop:energy-constant-iff-ph} for its complex geometric consequences, layed out in the following corollary shown in \S\ref{subsec:pf-cor-hol-structure}. Given a section $\phi$ of some holomorphic line bundle $L$ over a complex manifold $X$, we denote by $\mathbb{V}(\phi)$ the analytic subspace of $X$ defined by the vanishing of $\phi$.
\begin{corollary}\label{cor:hol-structure}
    Let $(M,\psi)$ be as in Proposition \ref{prop:energy-constant-iff-ph}(\ref{item:negative}). If $\mathcal{S}\subset\Teich_g$ is a connected complex submanifold contained in a level set of $\mathrm{E}_\psi$, then we have the following.
    \begin{enumerate}
        \item\label{item:cor-hol-constr-bundle} The bundle $E=F_\psi^* TM\otimes\mathbb{C}$ equipped with the $(0,1)$-part of the Levi--Civita connection on $M$ and the pullback metric from $M$ is a holomorphic Hermitian vector bundle over $\Curve_\mathcal{S}$.
        \item\label{item:cor-hol-constr-line} There exists a holomorphic line subbundle $L\leq E$ that inherits the Hermitian structure from $E$, such that $\partial F_\psi$ is an $L$-valued $(1,0)$ form that is closed with respect to the Chern connection.
        \item\label{item:cor-hol-constr-foliation} There is a (complex) codimension two subset $\Delta\subset\Curve_\mathcal{S}$ and a holomorphic codimension $1$ foliation $\mathcal{F}$ on $\Curve_\mathcal{S}\setminus\Delta$, such that $T\mathcal{F}=\ker\partial F_\psi$ on $\Curve_\mathcal{S}\setminus\mathbb{V}(\partial F_\psi)$ and such that the connected components of $\mathbb{V}(\partial F_\psi)\setminus \Delta$ are leaves of $\mathcal{F}$.
    \end{enumerate}
\end{corollary}
The harder ``only if'' direction of Proposition \ref{prop:energy-constant-iff-ph} will be shown in \S\ref{subsec:pf-prop-ph}. We show the ``if'' direction in \S\ref{subsec:pf-lm-only-if}. The proof of both directions of Proposition \ref{prop:energy-constant-iff-ph} depends on the computations in \cite{Tosic2023}, that we recall and slightly rephrase in \S\ref{subsec:dg}.
\subsection{Vector fields on the universal curve}\label{subsec:vf}
We now rephrase \cite[proof of Theorem 3]{Toledo2011} and \cite[Theorem 1.9]{Tosic2023} in differential geometric language that is more suited to our applications here. Let $J$ be the almost complex structure on $\mathcal{V}_g$ and let $\partial^{\mathrm{vert.}}$ and $\bar{\partial}^{\mathrm{vert.}}$ be the differential operators related to the Riemann surface structure on the fibres of $\pi$. Given $t\in\Teich_g$, let $X_t=\pi^{-1}(t)\subset\Curve_g$ be the Riemann surface that corresponds to $t$.
\begin{lemma}\label{lm:correspondence}
    There is a correspondence between smooth type $(1,0)$ vector fields on $X_t\subset\Curve_g$ that are lifts of some fixed vector in $T_{t}^{1,0}\Teich_g$, and smooth Beltrami forms on $X_t$, given by
    \begin{align}\label{eq:correspondence}
        V &\longrightarrow \mu=\overline{\bar{\partial}^{\mathrm{vert.}}V}.\tag{$\star$}
    \end{align}
    Moreover, in this case $\omega\circ\mathcal{L}_{V}J=2i\mu\omega^{1,0}$ for any $1$-form $\omega$ on $X_t$.
\end{lemma}
\begin{proof}
    It suffices to construct the inverse to the map (\ref{eq:correspondence}). We will use the computations in \cite{Markovic2024,Tosic2023}. We set $X=X_t$ throughout this proof.
    \par The point of view taken in \cite{Tosic2023} is to consider a fixed topological surface $\Sigma_g$, and then consider a family of almost-complex structures $J_t$ on $\Sigma_g$, parameterized by points $t\in\Teich_g$. To use the formulae from \cite{Tosic2023}, we therefore need to identify the fibres of $\Curve_\mathcal{S}\to\mathcal{S}$. Let $\mu$ be a smooth Beltrami differential on the marked Riemann surface $X\in\Teich_g$. As in \cite[\S 2.1]{Markovic2024}, we identify different fibres via solutions to the Beltrami equation $f^{t\mu}:X\to X^{t\mu}$, where $X^{t\mu}$ is a different Riemann surface, $t$ is a small real parameter, and 
    \begin{align*}
        t\mu=\frac{\bar{\partial}f^{t\mu}}{\partial f^{t\mu}}.
    \end{align*}
    We therefore define $V_\mu$ to be $\frac{\partial f^{t\mu}}{\partial t}$. We also let $V_\mu^{\mathbb{C}}=\frac{1}{2}\left(V_\mu-iV_{i\mu}\right)$. We will show that $\mu\to V_\mu^\mathbb{C}$ provides the inverse to (\ref{eq:correspondence}).
    \par Rephrasing \cite[Claim 3.2]{Markovic2024} or \cite[Claim 3]{Tosic2023}, we get the following claim. 
\begin{claim}\label{claim:jdot}
    For any 1-form $\omega$ on $X$ and any smooth Beltrami differential $\mu$ on $X$, we have
    \begin{gather*}
        \omega\circ\mathcal{L}_{V_\mu}J=2i\left(\mu\omega^{1,0}-\bar{\mu}\omega^{0,1}\right),\text{ and }\omega\circ\mathcal{L}_{V_\mu^\mathbb{C}}J=2i\mu\omega^{1,0}.
    \end{gather*} 
\end{claim}
\begin{proof}
    The first equality is equivalent to both \cite[Claim 3]{Tosic2023} and \cite[Claim 3.2]{Markovic2024}. The second equality follows from the first by an easy computation,
    \begin{align*}
        \omega\circ\mathcal{L}_{V_\mu^{\mathbb{C}}}J&=\frac{1}{2}\omega\circ\left(\mathcal{L}_{V_\mu}J-i\mathcal{L}_{V_{i\mu}}J\right)\\
        &=i\left(\mu\omega^{1,0}-\bar{\mu}\omega^{0,1}-i\left(i\mu\omega^{1,0}+i\bar{\mu}\omega^{0,1}\right)\right)\\
        &=2i\mu\omega^{1,0}.
    \end{align*} 
\end{proof}
Working in a local holomorphic coordinate $z$ on $X$ and setting $\omega=dz$, we see that 
\begin{align*}
    \left(\mathcal{L}_{V_\mu^\mathbb{C}}J\right)(\bar{\partial}_z)=2i\mu \partial_z.
\end{align*}
Here we abused notation slightly to denote by $\mu$ the coordinate expression for the Beltrami form in $z$. On the other hand, 
\begin{align*}
    \left(\mathcal{L}_{V_\mu^\mathbb{C}}J\right)(\bar{\partial}_z)&=-i\mathcal{L}_{V_\mu^\mathbb{C}}\bar{\partial}_z-J\mathcal{L}_{V_\mu^\mathbb{C}}\bar{\partial}_z\\
    &=-2i [V_\mu^\mathbb{C}, \bar{\partial}_z]^{1,0}\\
    &=-2i\left(\nabla_{V_\mu^\mathbb{C}}\bar{\partial}_z-\nabla_{\bar{z}}V_\mu^\mathbb{C}\right)^{1,0}\\
    &=2i\bar{\partial}V_\mu^{\mathbb{C}}.
\end{align*}
In particular, $\bar{\partial} V_\mu^\mathbb{C}=\mu$. This shows that (\ref{eq:correspondence}) is surjective. Injectivity follows immediately from the non-existence of holomorphic vector fields on $X$ without poles.
\end{proof}
\subsection{Vanishing of the Laplacian of the energy}\label{subsec:dg} Denote by $R^M$ the Riemann curvature tensor on $M$.
\begin{lemma}\label{lm:toledo}
    Fix $t_0\in\mathcal{T}_g$. Let $V$ be a vector field on $\mathcal{V}_g$ of type $(1,0)$ transverse to the fibres of $\pi$, such that $V$ is on $\pi^{-1}(t_0)$ a lift of a vector $\mu\in T_{t_0}\mathcal{T}_g$. Then we have 
    \begin{align}\label{eq:lm-toledo}
        \bar{\partial}^{\mathrm{vert.}}d_VF=-\frac{i}{2}\partial^\mathrm{vert.}F\circ\mathcal{L}_V J\text{ and }R^M(d_VF, \partial^\mathrm{vert.} F)=0
    \end{align}
    on $\pi^{-1}(t_0)$ if and only if the Laplacian of the energy $\mathrm{E}_\psi$ vanishes in the direction $\mu$.
\end{lemma}
\begin{proof} 
    Set $X=\pi^{-1}(t_0)$. We adopt the notation $V_\mu^\mathbb{C}$ from the proof of Lemma \ref{lm:correspondence} for the inverse of (\ref{eq:correspondence}).
By \cite[Theorem 1.9]{Tosic2023}, we see that the Laplacian of $\mathrm{E}_\psi$ vanishes in the direction defined by a smooth Beltrami form $\mu$ on $X\in\Teich_g$, if and only if 
\begin{align}\label{eq:1}
    \mu\partial^{\mathrm{vert.}}F=\bar{\partial}^{\mathrm{vert.}}\xi\text{ and }R^M(\xi, \partial^{\mathrm{vert.}}F)=0,
\end{align} 
for $\xi=\frac{1}{2}\left(\dot{F}^\mu-i\dot{F}^{i\mu}\right)=d_{V_\mu^{\mathbb{C}}}F$. Here by $\dot{F}^\mu$ denotes the derivative of $F\circ f^{t\mu}$ at $t=0$. Combining with Claim \ref{claim:jdot}, we get that (\ref{eq:1}) is equivalent to 
\begin{align*}
    -\frac{i}{2}\partial^{\mathrm{vert.}}F\circ\mathcal{L}_{V_\mu^\mathbb{C}}J=\bar{\partial}^{\mathrm{vert.}}d_{V_\mu^{\mathbb{C}}}F\text{ and }R^M(d_{V_\mu^\mathbb{C}}F, \partial^{\mathrm{vert.}}F)=0.
\end{align*}
Since any smooth vector field of type $(1,0)$ on $X\subset\Curve_g$ which is transverse to $X$ and projects under $\pi$ to some fixed $[\mu]\in T_X\Teich_g$ can be written in the form $V_\mu^{\mathbb{C}}$ for some smooth Beltrami form $\mu$ on $X$, the result is shown.
\end{proof}
\subsection{``If'' direction}\label{subsec:pf-lm-only-if}
We first claim that if (\ref{eq:lm-toledo}) holds for some vector field $V$, then it also holds for $\tilde{V}=fV+W$, where $W$ is any type $(1,0)$ (vertical) vector field on $X$ and $f:X\to\mathbb{C}$ is any smooth function. This is clear for the second equation in (\ref{eq:lm-toledo}). We show this claim for the first equation in (\ref{eq:lm-toledo}) in two parts. We first check the first equation in (\ref{eq:lm-toledo}) for $V=W$,
\begin{align*}
    \partial^{\mathrm{vert.}}F\circ\mathcal{L}_WJ&=\mathcal{L}_W\left(\partial^{\mathrm{vert.}}F\circ J\right)-\left(\mathcal{L}_W\partial^{\mathrm{vert.}}F\right)\circ J\\
    &=2i\mathcal{L}_W\partial^{\mathrm{vert.}}F\circ \frac{\mathrm{id}+iJ}{2}\\
    &=2i d^{\mathrm{vert.}}\left(\partial_W F\right)\circ\frac{\mathrm{id}+iJ}{2}=2i\bar{\partial}^\mathrm{vert.}d_W F.
\end{align*}
Here we used Cartan's magic formula and the fact that $d^{\mathrm{vert.}}\partial^{\mathrm{vert.}}F=0$ by harmonicity in going from the second to the third line. We now check how (\ref{eq:lm-toledo}) transforms for $\tilde{V}=fV$. We have 
\begin{align*}
    \bar{\partial}^{\mathrm{vert.}}d_{fV}F=f\bar{\partial}^{\mathrm{vert.}}d_VF+d_VF \bar{\partial}f, 
\end{align*}
and 
\begin{align*}
    \partial^{\mathrm{vert.}}F\circ\mathcal{L}_{fV}J&=f\partial^{\mathrm{vert.}}F\circ\mathcal{L}_VJ+2i \partial^{\mathrm{vert.}}F\circ\left(V\bar{\partial}f\right)\\
    &=f\partial^{\mathrm{vert.}}F\circ\mathcal{L}_Vf+2i d_VF\bar{\partial}f.
\end{align*}
where we used the following. 
\begin{claim}
    If $V$ is a type $(1,0)$ vector field, then \[\mathcal{L}_{fV}J=f\mathcal{L}_VJ+2i V\bar{\partial}f.\]
\end{claim}
\begin{proof}
    This follows from 
    \begin{align*}
        \mathcal{L}_{fV}(J)(W)&=\mathcal{L}_{fV}(JW)-J\mathcal{L}_{fV}W=[fV, JW]-J[fV, W]\\
        &=f[V, JW]-V d_{JW}f-J\left(f[V,W]-Vd_Wf\right)\\
        &=f\mathcal{L}_V(J)(W)+V \left(-d_{JW}f+i d_Wf\right)\\
        &=f\mathcal{L}_V(J)(W)+2iV\left(\bar{\partial}f\right)(W).
    \end{align*}
\end{proof}
\par We now let $(M,\psi)$ be as in Proposition \ref{prop:energy-constant-iff-ph} and suppose that $F_\psi:\Curve_\mathcal{S}\to M$ is pluriharmonic with $R^M(dF_\psi\wedge dF_\psi)=0$. By the previous argument, it suffices to check (\ref{eq:lm-toledo}) locally for any type $(1,0)$ vector field $V$ on $X$ transverse to $X$. We fix a point in $\Curve_\mathcal{S}$ and use local coordinates to check (\ref{eq:lm-toledo}) at that point. Let $(z,t_1,t_2,...,t_n)$ be holomorphic local coordinates on $\mathcal{V}_S$ and $(t_1,t_2,...,t_n)$ be local holomorphic coordinates for $\mathcal{S}$, such that $\pi$ is locally represented by $(z,t_1,t_2,...,t_n)\to (t_1,t_2,...,t_n)$. For $V=\frac{\partial}{\partial t_i}$, the first equation in (\ref{eq:lm-toledo}) holds, since both sides vanish. Similarly, since $R^M(\partial^{\mathrm{vert.}} F_\psi\wedge \partial F_\psi)=0$, we see that 
\begin{align*}
    \sum_{j=1}^n R^M\left(\frac{\partial F_\psi}{\partial z}, \frac{\partial F_\psi}{\partial t_j}\right)dz\wedge dt_j=0,
\end{align*}
and thus $R^M\left(\partial^{\mathrm{vert.}}F, d_{V}F\right)=0$. Therefore by Lemma \ref{lm:toledo} the Laplacian of $\mathrm{E}_\psi$ vanishes identically on $\mathcal{S}$. From the computations in \cite[\S 3.3]{Tosic2023} and \cite[Claim 3]{Tosic2023}, it is easy to see that if the Laplacian of $\mathrm{E}_\psi$ vanishes in the complex direction defined by $\mu$, then $d_{\mu}\mathrm{E}_\psi=0$. Thus $\mathrm{E}_\psi$ is constant on $\mathcal{S}$.
\subsection{``Only if'' direction}\label{subsec:pf-prop-ph}
Note that pluriharmonicity is a local claim, so we may without loss of generality assume that $\mathcal{S}$ is a small holomorphic disk around $t_0\in\mathcal{T}_g$. Since $\pi$ is a holomorphic submersion, we may work locally in coordinates $(z, t)$ on $\mathcal{V}_g$ and $t$ on $\mathcal{S}$ such that $\pi$ is represented locally by the projection $(z,t)\to t$ and such that $t_0$ corresponds to $t=0$. Let the Levi form of $F$ be 
\begin{align*}
    \partial\bar{\partial}F=F_{z\bar{z}}dz\wedge d\bar{z}+F_{t\bar{z}}dt\wedge d\bar{z}+F_{z\bar{t}}dz\wedge d\bar{t}+F_{t\bar{t}}dt\wedge d\bar{t}.
\end{align*}
We will also use the notation $dF=F_tdt+F_{\bar{t}}d\bar{t}+F_zdz+F_{\bar{z}}d\bar{z}$.
\par Since $V=\frac{\partial}{\partial t}$ is a holomorphic vector field, its flow preserves the almost complex structure $J$ on $\mathcal{V}_g$. In particular, from the first equation in the conclusion of Lemma \ref{lm:toledo}, we see that $F_{t\bar{z}}=0$. It immediately follows that $F_{z\bar{t}}=0$.
From the harmonicity of $F$, we see that $F_{z\bar{z}}=0$. Thus 
\begin{align*}
    F_{t\bar{t}}dt\wedge d\bar{t}=\partial \bar{\partial}F=\partial \left(F_{\bar{t}}d\bar{t}\right)=\nabla_t F_{\bar{t}} dt\wedge d\bar{t}.
\end{align*}
It is well-known that the Hessian is symmetric, so we have $\nabla_t F_{\bar{t}}=\nabla_{\bar{t}} F_t$. We therefore have 
\begin{align*}
    \nabla_z F_{t\bar{t}}&=\nabla_z \nabla_t F_{\bar{t}}=\nabla_t \nabla_z F_{\bar{t}}+R^M(F_*\partial_z, F_*\partial_t)F_{\bar{t}}\\
    &=\nabla_t F_{z\bar{t}}+R^M(\partial^{\mathrm{vert.}}F, d_VF)d_{\bar{V}}F=0,
\end{align*}
where the first term is the final sum vanishes since $F_{z\bar{t}}=0$ and the second term vanishes from the second equation in the conclusion of Lemma \ref{lm:toledo}. We analogously have 
\begin{align*}
    \nabla_{\bar{z}} F_{t\bar{t}}=\nabla_{\bar{z}} \nabla_{\bar{t}} F_t=\nabla_{\bar{t}}\nabla_{\bar{z}}F_t+R^M(F_*\partial_{\bar{z}}, F_*\partial_{\bar{t}})F_t=0.
\end{align*}
In particular, $dF_{t\bar{t}}=0$. 
\begin{claim}
    We have $R^M(F_{t\bar{t}}, F_{z})=0$.
\end{claim}
\begin{proof}
    If $M$ has strictly negative Hermitian sectional curvature, then $F_t\wedge F_z=0$, so differentiating we get $F_{t\bar{t}}\wedge F_z+F_t\wedge F_{z\bar{t}}=0$, so $F_{t\bar{t}}\wedge F_z=0$ and the result follows. 
    \par If $M$ is locally symmetric of non-positive curvature, then 
    \begin{align*}
        0=\nabla_{\bar{t}} R^M(F_t, F_z)=R^M(F_{t\bar{t}}, F_z)+R^M(F_{t}, F_{z\bar{t}})=R^M(F_{t\bar{t}}, F_z).
    \end{align*}
\end{proof}
\par Now let $W$ be an arbitrary smooth type $(1,0)$ vector field on $\pi^{-1}(\mathcal{S})$, that is a lift of a nowhere vanishing holomorphic vector field on $\mathcal{S}$. Then in the local cooridnates $(z,t)$, we see that $W=V+\alpha(z,t)\partial_z$, and hence 
\begin{align*}
    F_{W\bar{W}}=F_{t\bar{t}}+\abs{\alpha}^2 F_{z\bar{z}}+\alpha F_{z\bar{t}}+\bar{\alpha}F_{t\bar{z}}=F_{t\bar{t}},
\end{align*}
and hence we see that 
\begin{align}\label{eq:F_WW}
    d F_{W\bar{W}}=0\text{ and }R^M(F_{W\bar{W}},\partial^{\mathrm{vert.}}F)=0.
\end{align}
\par If $M$ is locally symmetric of non-positive curvature, then by a result of Sunada \cite{Sunada1979}, the harmonic map realizing $\psi_*:\pi_1(\Sigma_g)\to \mathrm{Isom}(\tilde{M})$ is not unique unless $F_{W\bar{W}}=0$. However uniqueness follows from the fact that $\psi_*$ has a trivial stabilizer. Therefore $F_{W\bar{W}}=0$ and the result is shown in this case.
\par Suppose now that $M$ has negative Hermitian sectional curvature, and that $\psi$ is not homotopic into a graph. We want to show that $F_{W\bar{W}}=0$, so since $dF_{W\bar{W}}=0$ it suffices to show that $F_{W\bar{W}}$ vanishes somewhere on every fibre of $\pi:\pi^{-1}(\mathcal{S})\to\mathcal{S}$. Suppose therefore that $F_{W\bar{W}}$ is nowhere zero on $\pi^{-1}(t_0)$, and set $f=F|_{\pi^{-1}(t_0)}$.
\par Let $L$ be the rank one subbundle of $f^*TM$ generated by $F_{W\bar{W}}$. Then the second equation in (\ref{eq:F_WW}) implies that $\partial f$ takes values in $L\otimes\mathbb{C}$. In particular, we see that the image of $df$ is contained in $L$. Therefore we get that 
\begin{align*}
    df=\omega F_{W\bar{W}},
\end{align*}
where $\omega$ is a harmonic 1-form on $\pi^{-1}(t_0)$. In particular, $df$ has rank at most $1$, and hence by a result of Sampson \cite{Sampson1978}, the image of $f$ is contained in a geodesic arc. Since the domain of $f$ is closed, this arc is a closed curve, which contradicts the assumption on $\psi$.
\subsection{Proof of Corollary \ref{cor:hol-structure}}\label{subsec:pf-cor-hol-structure}
We first show (\ref{item:cor-hol-constr-bundle}).
Let $\bar{\partial}^E$ be the $(0,1)$-part of the pullback of the Levi--Civita connection on $TM$ to $E:=F_\psi^*TM\otimes\mathbb{C}$. We work in local coordinates $(z,t_1,t_2,...,t_m)$ on $\Curve_\mathcal{S}$ and $(t_1,t_2,...,t_m)$ on $\mathcal{S}$, such that the projection $\pi$ locally takes the form $(z,t_1,t_2,...,t_m)\to (t_1,t_2,...,t_m)$.
\par By Proposition \ref{prop:energy-constant-iff-ph}, we see that $R^M(\partial_z F_\psi, \partial_{t_i}F_\psi)=0$. Since $M$ has negative Hermitian sectional curvature, this implies that $\partial_z F_\psi\wedge \partial_{t_i}F_\psi=0$. On the set of points where $\partial_z F_\psi\neq 0$, this implies that $d_V F_\psi\wedge d_W F_\psi=0$, for any two type $(1,0)$ vectors $V, W$. Note that $\partial_z F_\psi$ has isolated zeros on any $\pi^{-1}(t)$ for $t\in\mathcal{S}$ \cite[pp. 10]{Schoen1997}. By continuity, we therefore have $d_VF_\psi\wedge d_WF_\psi=0$ for any two $(1,0)$ vectors $V, W$ at any point in $\Curve_\mathcal{S}$. In particular, $R^M(\partial F_\psi\wedge\partial F_\psi)=0$.
\par Conjugating, we see that $R^M(\bar{\partial}F_\psi\wedge\bar{\partial}F_\psi)=0$, and since 
\begin{align*}
    \bar{\partial}^E\circ\bar{\partial}^E=R^M(\bar{\partial}F_\psi\wedge\bar{\partial}F_\psi)=0,
\end{align*}
finishing the proof of (\ref{item:cor-hol-constr-bundle}) by the Koszul--Malgrange theorem.
\par We now turn to (\ref{item:cor-hol-constr-line}). From Proposition \ref{prop:energy-constant-iff-ph}, and the strict negativity of the Hermitian sectional curvature of $M$, we see that 
\begin{align*}
    \partial F_\psi|_{\ker\pi_*}\wedge\partial F_\psi=0.
\end{align*}
This implies that over a point where $\partial F_\psi|_{\ker\pi_*}\neq 0$, we have the stronger $\partial F_\psi\wedge\partial F_\psi=0$. Thus by continuity we see that $\partial F_\psi\wedge\partial F_\psi=0$ on all of $\Curve_\mathcal{S}$. We now get a line bundle $L\leq E$ over $\Curve_\mathcal{S}\setminus\mathbb{V}(\partial F_\psi)$ generated by the image of $\partial F_\psi$. By \cite[Theorem 4 (2)]{Toledo2011}, we see that $L$ extends to $\Curve_\mathcal{S}$. Finally, since the Chern connection on $E$ is by construction the pullback of the Levi--Civita connection on $M$, we have 
\begin{align*}
    d^{\nabla}\partial F_\psi=\bar{\partial}\partial F_\psi=0.
\end{align*}
The Chern connection on $L$ is the orthogonal projection of the Chern connection on $E$ to $L$, so $\partial F_\psi\in\Omega^1(L)$ is closed as well.
\par Finally, we show (\ref{item:cor-hol-constr-foliation}). Let $U\subset\Curve_\mathcal{S}$ be an open set over which $L$ is trivial, and let $\tau:L|_U\to\mathbb{C}$ be some (holomorphic) trivialization. Let $\Gamma$ be the local connection 1-form with respect to this trivialization. Since $\partial F_\psi$ is closed, we see that 
\begin{align*}
    d\tau(\partial F_\psi)+\Gamma\wedge\tau(\partial F_\psi)=0.
\end{align*}
In particular, the distribution defined by vanishing of the components of $\tau(\partial F_\psi)$ is integrable by the Frobenius theorem \cite[Theorem II.1.1]{Bryant1991} wherever its rank is locally constant. It follows that there exists a holomorphic codimension $1$ foliation $\mathcal{F}$ over $\Curve_\mathcal{S}\setminus\mathbb{V}(\partial F_\psi)$ with $T\mathcal{F}=\ker\partial F_\psi$. 
\par Suppose that $\dim\mathcal{S}=d$, so that $\dim\Curve_\mathcal{S}=d+1$. We write $\mathcal{A}=\mathbb{V}(\partial F_\psi)$ as a disjoint union 
\begin{align*}
    \mathcal{A}=\mathcal{A}^\mathrm{sng.}\cup\mathcal{A}^\mathrm{reg.}=\mathcal{A}^\mathrm{sng.}\cup\bigcup_{i=0}^d \mathcal{A}_i,
\end{align*}
where \begin{enumerate}
    \item $\mathcal{A}^\mathrm{sng.}$ is the singular part of $\mathcal{A}$, 
    \item $\mathcal{A}^\mathrm{reg.}=\mathcal{A}\setminus\mathcal{A}^\mathrm{sng.}$ is the regular (smooth) part of $\mathcal{A}$, and 
    \item $\mathcal{A}^\mathrm{reg.}=\bigcup_{i=0}^d \mathcal{A}_i$ is the decomposition of $\mathcal{A}^\mathrm{reg.}$ into the union of its $i$-dimensional connected components $\mathcal{A}_i$.
\end{enumerate}
We now let $\Delta=\mathcal{A}^\mathrm{sng.}\cup\bigcup_{i=0}^{d-1}\mathcal{A}_i\cup\Delta'\cup\Delta''$, where \begin{enumerate}
    \item $\Delta'$ is the subset of $\mathcal{A}_d$ where the holomorphic map $\pi|_{\mathcal{A}_d}:\mathcal{A}_d\to \mathcal{S}$ is not a local biholomorphism, and 
    \item $\Delta''$ is the subset of $\mathcal{A}_d$ where the order of vanishing of $\partial^\mathrm{vert.}F_\psi$ is not locally constant.
\end{enumerate}
It follows from classical complex geometry that $\dim\Delta\leq d-1$. Note that $\mathcal{F}$ is defined on $\Curve_\mathcal{S}\setminus\mathcal{A}$, and we now show how to extend it to $\mathcal{A}\setminus\Delta=\mathcal{A}_d\setminus(\Delta'\cup\Delta'')$.
\par Fix some point $x\in\mathcal{A}_d\setminus(\Delta'\cup\Delta'')$. Pick a coordinate system $\Phi=(z,t_1,t_2,...,t_d):U\to \mathbb{C}^{d+1}$ on an open set $x\in U\subset\Curve_\mathcal{S}$ such that $\Phi\left(\mathcal{A}_d\cap U\right)=\mathbb{V}(z)$ and $\Phi(x)=(0,0,...,0)$. By choice of $\Delta'$, we may suppose without loss of generality that there is a coordinate system $\Theta=(t_1,t_2,...,t_d):\pi(U)\to\mathbb{C}^d$ such that 
\begin{align*}
    \Theta\circ\pi\circ\Phi^{-1}(z,t_1,t_2,...,t_d)=(t_1,t_2,...,t_d).
\end{align*}
We also choose a trivialization $\tau:L|_U\to\mathbb{C}$. Observe that $\tau(\partial F_\psi)$ vanishes along $\{z=0\}$, so can be written locally as 
\begin{align*}
    \Phi_*\tau(\partial F_\psi)=z^n f dz+\sum_{j=1}^d g_j dt_j.
\end{align*}
By choice of $\Delta''$, we may assume that $f(0,0,...,0)\neq 0$. Shrink $U$ if necessary so that $f\neq 0$ on $\Phi(U)$. Let the local connection 1-form for $L$ in the trivialization $\tau$ be $\Gamma=\Gamma_0 dz+\sum_{j=1}^d\Gamma_j dt_j$. The fact that $\partial F_\psi$ is closed implies that
\begin{align}\label{eq:closed-local}
    z^{n} \left(\frac{\partial f}{\partial t_j}+\Gamma_j f\right)=\frac{\partial g_j}{\partial z}+\Gamma_0 g_j\text{ for }j\geq 1.
\end{align}
Note that, by assumption, $\partial F_\psi$ vanishes along $\mathbb{V}(z)$. In particular, $g_j(0,t_1,t_2,...)=0$. Taking derivatives in $z$ of (\ref{eq:closed-local}), we get $\frac{\partial^k g_j}{\partial z^k}=0$ along $\mathbb{V}(z)$ for $0\leq k\leq n$. Thus we can write $g_j=z^{n+1} h_j$, and hence
\begin{align*}
    \Phi_*\tau(\partial F_\psi)=z^{n}\left(fdz+z\sum_{j\geq 1}h_j dt_j\right).
\end{align*}
We now observe that $\Phi_*\mathcal{F}$ coincides with $\ker\omega$ away from $\mathbb{V}(z)$, where 
\begin{align*}
    \omega=fdz+z\sum_{j\geq 1}h_j dt_j.
\end{align*}
We will show that the distribution defined by $\ker\omega$ is integrable in $\Phi(U)$, which implies that $\mathcal{F}$ extends over $\mathcal{A}_d\setminus(\Delta'\cup\Delta'')$, and that moreover every connected component of $\mathcal{A}_d\setminus(\Delta'\cup\Delta'')$ is a leaf.
\par Note that 
\begin{align*}
    d\omega=d(z^{-n}\Phi_*\tau(\partial F_\psi))&=-\Gamma\wedge\omega -n\frac{dz}{z}\wedge\omega\\
    &=-\Gamma\wedge\omega-n\frac{\omega-z\sum_{j\geq 1}h_j dt_j}{zf}\wedge\omega \\
    &=\left(\sum_{j\geq 1}\frac{h_j}{f}dt_j-\Gamma\right)\wedge\omega.
\end{align*}
By the Frobenius integrability theorem, it follows that $\ker\omega$ is integrable, which concludes the proof of Corollary \ref{cor:hol-structure} (\ref{item:cor-hol-constr-foliation}).
\section{Schiffer varieties}\label{sec:schiffer}
For a representation $\rho:\pi_1(\Sigma_g)\to\mathrm{PSL}(2,\mathbb{R})$, let $\Lambda_\rho\subset\Teich_g$ be the set of marked Riemann surfaces $X$ whose universal cover $\tilde{X}$ admits a $\rho$-equivariant holomorphic map $\tilde{X}\to\mathbb{H}$. Let $\tilde{\Curve}_g$ be the universal cover of $\Curve_g$, and let $\tilde{\Curve}_\mathcal{S}$ be the preimage of $\Curve_\mathcal{S}$ under this universal cover, where $\mathcal{S}\subset\Teich_g$ is a submanifold. Our main result in this section is the following. 
\begin{proposition}\label{prop:schiffer}
    For any representation $\rho:\pi_1(\Sigma_g)\to\mathrm{PSL}(2,\mathbb{R})$ of positive Euler class, the subset $\Lambda_\rho\subset\Teich_g$ is a locally finite union of connected closed submanifolds of $\Teich_g$. Moreover, there exists a $\rho$-equivariant holomorphic map $F:\tilde{\Curve}_{\Lambda_\rho}\to\mathbb{H}$.
\end{proposition}
We use the techniques of \cite{Tosic2023} to show Proposition \ref{prop:schiffer}. We now briefly outline the proof of Proposition \ref{prop:schiffer}. Let $\mathcal{E}(g,k)$ be the moduli space of all harmonic maps $\tilde{X}\to\mathbb{H}$, equivariant with respect to a representation of Euler class $k$, for a varying marked Riemann surface $X\in\Teich_g$. This space is naturally a holomorphic fibration over Teichm\"uller space, that we will describe in \S\ref{subsec:nah} and show in Appendix \ref{appendix}. Let $\mathcal{Z}(g,k)\subset\mathcal{E}(g,k)$ be the subset that consists of holomorphic maps. 
\par Define a distribution $\mathcal{K}$ on $\mathcal{E}(g,k)$ as follows. The non-abelian Hodge correspondence provides a diffeomorphism $\mathrm{NAH}:\mathcal{E}(g,k)\to\Teich_g\times\chi_{g,k}$. Here $\chi_{g,k}$ denotes the connected component of the representation variety $\pi_1(\Sigma_g)\to\mathrm{PSL}(2,\mathbb{R})$ that consists of representations with Euler class $k$. Now define a distribution $\tilde{\mathcal{K}}$ on $\Teich_g\times\chi_{g,k}$ as $\tilde{\mathcal{K}}_{(X, \rho)}=\{(\mu,0)\in T_X\Teich_g\oplus T_\rho\chi_{g,k}:\Delta_\mu\mathrm{E}_\rho(X)=0\}$. We then let $\mathcal{K}=\mathrm{NAH}^*\tilde{\mathcal{K}}$. Note that this distribution was previously studied by the author in \cite{Tosic2023}.
\par We then show that $\mathcal{K}$ restricts to an integrable distribution on $\mathcal{Z}(g,k)$. Let the corresponding foliation of $\mathcal{Z}(g,k)$ be $\mathcal{F}$. Then the connected components of $\Lambda_\rho$ are the leaves of $\mathrm{NAH}_*\mathcal{F}$ contained in $\Teich_g\times\{\rho\}$. Local finiteness of this set of leaves then follows easily from the fact that $\Lambda_\rho$ is a real analytic subset of $\Teich_g$.
\subsection{Non-abelian Hodge over a moving Riemann surface}\label{subsec:nah}
We closely follow \cite[\S 10]{Hitchin1987}. 
\par An $\mathrm{SL}(2,\mathbb{R})$-Higgs bundle over a marked Riemann surface $X\in\Teich_g$ is a triple $(L,\alpha,\delta)$, where $L$ is a holomorphic line bundle, $\alpha\in H^0(L^2K_X)$ and $\delta\in H^0(L^{-2}K_X)$. The corresponding Higgs bundle is $E=L\oplus L^{-1}$ with the field 
\begin{align*}
    \phi=\begin{pmatrix}
        0 & \alpha \\
        \delta & 0
    \end{pmatrix}.
\end{align*} 
Suppose that $L$ is positive. Then this Higgs bundle is stable if and only if $\delta\neq 0$. Note that $\phi$ is gauge equvialent to $\begin{pmatrix}
    0 & c\alpha \\
    c^{-1}\delta & 0
\end{pmatrix}$ for any $c\in\mathbb{C}\setminus\{0\}$. Thus the map $(L,\alpha,\delta)\to (\mathrm{div}(\delta), \alpha\delta)$ provides an isomorphism between the moduli space of Higgs bundles over $X$ of Euler class $k=\deg L^2$, and the space of $(D,\Phi)\subset \Sym^{2g-2-k}X\times H^0(K_X^2)$ where $\mathrm{div}(\Phi)\geq D$. This was shown by Hitchin in \cite[\S 10]{Hitchin1987}.
\par This construction carries over to the case of a moving Riemann surface without much difficulty. In this section we state our most general result, and then we prove it for the sake of completeness in Appendix \ref{appendix}.
\par Recall that $\pi:\Curve_g\to\Teich_g$ is the universal curve over Teichm\"uller space $\Teich_g$, so that the fibre over $X\in\Teich_g$ is isomorphic to $X$. Define $\Sym^n\pi:\Sym^n\Curve_g\to\Teich_g$ to be the bundle over $\Teich_g$, whose fibre over $X$ is $\Sym^n X$. We also let $\mathcal{H}_g$ be the vector bundle over $\Teich_g$, such that the fibre over $X\in\Teich_g$ naturally parameterizes $H^0(K_X^2)$. This bundle is the $\pi$-pushforward of the square of the relative canonical bundle $K_\pi\to\Curve_g$.
\begin{definition}\label{dfn:egk}
    The space $\mathcal{E}(g,k)$ is the subhseaf of $(\Sym^n\pi)^*\mathcal{H}_g$ whose fibre over $X\in\Teich_g$ consists of pairs $(D,\Phi)\in\Sym^{2g-2-k}X\times H^0(K_X^2)$ such that $\mathrm{div}(\Phi)\geq D$.
\end{definition}
Using the Riemann--Roch theorem, it is easy to see that the subsheaf $\mathcal{E}(g,k)$ has constant rank, and is thus a vector bundle. Alternatively, in Appendix \ref{appendix} we will give a different construction of $\mathcal{E}(g,k)$ in terms of the more classical operations in complex geometry that shows immediately that it is a vector bundle. \par The non-abelian Hodge correspondence provides a diffeomorphism between the fibre $\mathcal{E}(g,k)_X$ and $\chi_{g,k}$, that varies smoothly with the underlying Riemann surface $X\in\Teich_g$. We state this as Theorem \ref{thm:nah} below, whose proof we include in Appendix \ref{appendix}, for completeness.
\begin{theorem}\label{thm:nah}
    The non-abelian Hodge correspondence provides a diffeomorphism $\mathrm{NAH}:\mathcal{E}(g,k)\to\Teich_g\times\chi_{g,k}$.
\end{theorem}
\subsection{Foliation}
The following lemma easily implies Proposition \ref{prop:schiffer}. Recall that $\mathcal{K}=\mathrm{NAH}^*\tilde{\mathcal{K}}$, where $\tilde{\mathcal{K}}$ is the distribution on $\Teich_g\times\chi_{g,k}$, defined on the slice $\Teich_g\times\{\rho\}$ as the kernel of the Levi form of $\mathrm{E}_\rho$. The distribution $\mathcal{K}$ is smooth by \cite[\S 4]{Tosic2023} and Theorem \ref{thm:nah}.
\begin{lemma}\label{lemma:integrable}
    The zero section $\mathcal{Z}(g,k)$ in $\mathcal{E}(g,k)$ admits a foliation $\mathcal{F}$ with $T\mathcal{F}|_{\mathcal{Z}(g,k)}=\mathcal{K}|_{\mathcal{Z}(g,k)}$. The leaves of $\mathcal{F}$ take the form $\mathrm{NAH}^{-1}(\Lambda\times\{\rho\})$, where $\rho:\pi_1(\Sigma_g)\to\mathrm{PSL}(2,\mathbb{R})$ is a non-elementray representation and $\Lambda\subset\Teich_g$ is a complex submanifold. For any such leaf, there exists a $\rho$-equivariant holomorphic map $F:\tilde{\Curve}_\Lambda\to \mathbb{H}$.
\end{lemma}
Note that Lemma \ref{lemma:integrable} clearly implies Proposition \ref{prop:schiffer}, apart from the claim about local finiteness. That follows from the following claim.
\begin{claim}
    For any $\rho:\pi_1(\Sigma_g)\to\mathrm{PSL}(2,\mathbb{R})$, the set $\Lambda_\rho$ is a real analytic subset of $\Teich_g$.
\end{claim}
\begin{proof}
    Fix an $X\in\Teich_g$, and let $f:\tilde{X}\to\mathbb{H}$ be the $\rho$-equivariant harmonic map. Denote by $\mathrm{eu}(\rho)$ the Euler class of a representation $\rho$. Observe that 
    \begin{align*}
        \mathrm{E}_\rho(X)-2\pi\mathrm{eu}(\rho)&=\int_X \left(\abs{\partial f}^2+\abs{\bar{\partial}f}^2\right)d\mathrm{area}_X-\int_X \left(\abs{\partial f}^2-\abs{\bar{\partial}f}^2\right)d\mathrm{area}_X \\
        &=2\int_X \abs{\bar{\partial}f}^2 d\mathrm{area}_X.
    \end{align*}
    Hence $\Lambda_\rho$ is precisely $\mathrm{E}_\rho^{-1}(2\pi\mathrm{eu}(\rho))$, and is thus real analytic as $\mathrm{E}_\rho$ is real analytic.
\end{proof}
The rest of this section is devoted to proving Lemma \ref{lemma:integrable}. We split the proof into three steps, that we show in the following three subsections. We first show that $\mathcal{K}$ is tangent to $\mathcal{Z}(g,k)$ in \S\ref{subsec:tangency}. Then we show that $\mathcal{K}$ is integrable over $\mathcal{Z}(g,k)$ in \S\ref{subsec:integrability}, and finally we show in \S\ref{subsec:holomorphy} that the the $\rho$-equivariant harmonic maps are holomorphic over the parts of $\tilde{\Curve}_g$ that project to leaves of $\mathcal{F}$.
\subsection{Tangency to the zero section}\label{subsec:tangency} The following is a simple consequence of the computations in \cite[\S 3.1]{Tosic2023}. 
\begin{lemma}
    Let $\rho:\pi_1(\Sigma_g)\to\mathrm{PSL}(2,\mathbb{R})$ be a non-elementary representation. Suppose that $S_0\in\Teich_g$ and $f_0:\tilde{S}_0\to \mathbb{H}$ is a holomorphic $\rho$-equivariant map from the universal cover $\tilde{S}_0$ of $S_0$, and let $((S_t, f_t):t\in\mathbb{D})$ be a complex disk of $\rho$-equivariant harmonic maps based at $S_0$ with direction $\mu\in T_{S_0}\Teich_g$ (see \cite[Definition 1.7]{Tosic2023}). Assume further that 
    \begin{align*}
        \mu\partial f=\bar{\partial}\xi\text{ and }\xi\wedge\partial f=0,
    \end{align*}  
    for some section $\xi$ of the bundle $f_0^*T\mathbb{H}\otimes\mathbb{C}$. Then if we denote by $\Phi_t$ the Hopf differential of $f_t$, we have $\Phi_t=O\left(\abs{t}^2\right)$ for small $t$.
\end{lemma}
\begin{proof}
    From \cite[Theorem 1.9]{Tosic2023}, we see that $\eta=\left.\frac{\partial f_t}{\partial t}\right|_{t=0}-\xi$ is a parallel section such that $\eta\wedge df_0=0$. However, since $\rho$ is non-elementary, $df_0$ has rank $2$ at some point. Therefore $\eta$ vanishes at this point, and since it is parallel, $\eta$ is identically zero. Thus $\xi=\left.\frac{\partial f_t}{\partial t} \right|_{t=0}$.
    \par Now observe that $\Phi_t=\langle \partial f_t\otimes\partial f_t\rangle$, where $\langle\cdot,\cdot\rangle$ is the complex bilinear extension of the pullback metric on $f_t^*T\mathbb{H}\otimes\mathbb{C}$. Therefore we have, setting $t=x+iy$, in the setting of \cite[\S 3.1]{Tosic2023}, 
    \begin{align*}
        \frac{\partial\Phi_t}{\partial x}&=\frac{\partial}{\partial x}\left\langle \left(df_t\circ\frac{\mathrm{id}-iJ_t}{2}\right)\otimes\left(df_t\circ\frac{\mathrm{id}-iJ_t}{2}\right) \right\rangle \\
        &=2\left\langle (2{\partial}\mathrm{Re}(\xi)+\mu\partial f_0-\bar{\mu}\bar{\partial}f_0)\otimes\partial f_0 \right\rangle\\
        &=2\left\langle (\partial\xi+\partial\bar{\xi}+\bar{\partial}\xi-\partial\bar{\xi})\otimes\partial f_0\right\rangle=2\langle d\xi\otimes\partial f_0\rangle\\
        &=2\mu\Phi_0+2\langle \partial\xi\otimes\partial f_0\rangle.
    \end{align*}
    Note that since $f_0$ is holomorphic, we have $\Phi_0=0$. Moreover, since $\xi\wedge\partial f_0=0$ and $f_0$ is holomorphic, it follows that $\xi$ takes values in $f_0^*T^{1,0}\mathbb{H}$. Since $f_0$ is holomorphic, it follows that $f_0^*T\mathbb{H}\otimes\mathbb{C}=f_0^*T^{1,0}\mathbb{H}\oplus f_0^*T^{0,1}\mathbb{H}$, and in particular the second fundamental form of $f_0^*T^{1,0}\mathbb{H}\leq f_0^*T\mathbb{H}\otimes\mathbb{C}$ vanishes. Therefore $\partial\xi$ also takes values in $f_0^*T^{1,0}\mathbb{H}$, and hence $\langle\partial\xi\otimes\partial f_0\rangle=0$. Thus $\frac{\partial \Phi_t}{\partial x}=0$, and an identical argument shows that $\frac{\partial\Phi_t}{\partial y}=0$, concluding the proof. 
\end{proof}
\subsection{Integrability}\label{subsec:integrability} By \cite[Theorem 1.1]{Tosic2023}, the distribution $\mathcal{K}$ is precisely $\mathrm{NAH}^*\ker d\mathcal{R}$, where $\mathcal{R}:\Teich_g\times\chi_{g,k}\to \chi_{g,k}$ is a certain map defined in \cite[pp. 2]{Tosic2023}. In particular, if we show that $\mathcal{K}$ has constant rank on $\mathcal{Z}(g,k)$, it follows immediately from Frobenius integrability theorem that $\mathcal{K}$ is integrable.
\par We only show this in the case of even Euler class. Odd Euler class can be handled as in \cite[\S 4]{Biswas2021} or \S\ref{subsubsec:odd}. Let $(E,\Phi)$ be the stable degree $0$ Higgs bundle over $X\in\Teich_g$ that corresponds to some point $p\in\mathcal{Z}$. From \cite[\S 10]{Hitchin1987}, it follows that $E=L\oplus L^{-1}$ for some holomorphic line bundle $L$ over $X$, such that 
\begin{align*}
    \Phi=\begin{pmatrix}
        0 & \phi\\
        0 & 0
    \end{pmatrix},
\end{align*}
where $\phi\in H^0(L^2K)$. Here $L$ is a negative bundle with degree that depends only on the Euler class. By \cite[Theorem 1.6]{Tosic2023}, $\mathcal{K}_{p}$ consists of directions $\mu$ such that 
\begin{align*}
    \mu\phi=\bar{\partial}\xi,
\end{align*}
for some $\xi$ that is a section of $L^2$. The existence of such a $\xi$ is equivalent to $[\mu\phi]$ vanishing in $H^1(L^2)$. By Serre duality, this is equivalent to $\int_S \mu\phi\theta=0$ for all $\theta\in H^0(KL^{-2})$. Thus 
\begin{align*}
    \dim \mathcal{K}_p&=3g-3-\dim \phi\otimes H^0(KL^{-2})=3g-3-h^0(KL^{-2})\\
    &=3g-3-(h^0(L^2)-\deg L^2+g-1)=2g-2+\deg L^2\\ &=2g-2-k,
\end{align*}
which is independent of the point in $\mathcal{Z}(g,k)$.
\subsection{Holomorphicity of the combined map}\label{subsec:holomorphy} Let $F:\tilde{\Curve}_\Lambda\to\mathbb{H}$ be the map which is harmonic and $\rho$-equivariant on the fibres of $\tilde{\Curve}_g\to\Teich_g$. By construction, $F$ is holomorphic on the fibres of $\tilde{\Curve}_g\to\Teich_g$. By Proposition \ref{prop:energy-constant-iff-ph}, we see that $F$ is pluriharmonic. 
\par We now work in local coordinates $(z,t_1,t_2,...,t_n)$ on $\tilde{\Curve}_\Lambda$ and $(t_1,t_2,...,t_n)$ on $\Lambda$, such that $\tilde{\Curve}_g\to\Teich_g$ is locally modelled on $(z,t_1,t_2,...,t_n)\to (t_1,t_2,...,t_n)$. It follows from Proposition \ref{prop:energy-constant-iff-ph} that 
\begin{align}\label{eq:collinear}
    F_z\wedge F_{t_i}=0.
\end{align}
However, since $\langle F_z\otimes F_z\rangle=0$ by conformality, it follows that $\langle F_{t_i}\otimes F_{t_j}\rangle=0$ and $\langle F_{t_i}\otimes F_z\rangle=0$ for any $1\leq i,j\leq n$, since by (\ref{eq:collinear}), $F_z$ and $F_{t_i}$ are colinear. Therefore $F$ is holomorphic when restricted to any complex disk in $\tilde{\Curve}_\Lambda$, and hence $F$ is holomorphic.
\section{Proof of Theorem \ref{thm:inframain}}\label{sec:pf-thm-main}
In this section, we prove Theorem \ref{thm:inframain}. Therefore let $M, \psi, \mathcal{S}$ be as in Theorem \ref{thm:inframain}, and $\Delta,\mathcal{F}$ as in Corollary \ref{cor:hol-structure}.
\par Recall from the outline that, to show Theorem \ref{thm:inframain}, we define an equivalence relation $\sim$ on $X_t$, such that the quotient map $q_t:X_t\to Y_t:=X_t/\sim$ is a non-constant holomorphic map between Riemann surfaces. We then show that the Riemann surfaces $Y_t$ can be taken to be isomorphic over $\mathcal{S}$.
\par The organization of this section is as follows. We define $\sim$ in \S\ref{subsec:holomorphic-equivalence-relation}. In \S\ref{subsec:semicontinuity}, we show the semicontinuity of the family $\{Y_t:t\in\mathcal{S}\}$. More precisely, we show in Lemma \ref{lm:semicontinuity} that for any $t\in\mathcal{S}$, there exists a neighbourhood $t\in\mathcal{U}\subset\mathcal{S}$ such that, for any $s\in\mathcal{U}$, there exists a non-constant holomorphic map $Y_t\to Y_s$. Using this observation, we finish the proof of Theorem \ref{thm:inframain} in \S\ref{subsec:finishing-thm-inframain}.
\subsection{Analytic equivalence relation}\label{subsec:holomorphic-equivalence-relation} 
We define the following equivalence relation $\sim$ on $X_t$, following Sagman \cite[Definition 3.1]{Sagman2021}. Define $X_t^\circ=X_t\setminus\mathbb{V}(\partial^\mathrm{vert.}F_\psi)$.
\begin{definition}\label{dfn:sim}
    We define an equivalence relation $\widehat{\sim}$ on $X_t^\circ$ as follows. For $z, w\in X_t^\circ$, we have $z\widehat{\sim} w$ if and only if there exist open neighbourhoods $U_z$ of $z$ and $U_w$ of $w$ in $X_t^\circ$, and a biholomorphism $\phi:U_z\to U_w$ such that $F_\psi\circ \phi=F_\psi$. We let $\sim$ be the relation obtained by taking the closure of $\widehat{\sim}\subset X_t\times X_t$.
\end{definition}
Note that by \cite[Defintion 3.1, Propositions 3.15 and 3.16]{Sagman2021}, ${\sim}$ is an equivalence relation, and there exists a Riemann surface strucutre on $X_t/{\sim}$ such that $X_t\to X_t/{\sim}$ is holomorphic. We define $Y_t:=X_t/{\sim}$, and let $q_t:X_t\to Y_t$ be the (holomorphic) quotient map. Moreover, $\sim\cap (X_t^\circ\times X_t^\circ)=\widehat{\sim}$.
\begin{remark}\label{remark:nathaniel-eq-sim}
    Note that \cite[Definition 3.1]{Sagman2021} differs slightly from Definition \ref{dfn:sim}. Let $\sim^\prime$ be the relation from \cite[Definition 3.1]{Sagman2021}. Then $\sim^\prime\cap (X_t^\circ\times X_t^\circ)=\widehat{\sim}$ and it is easy to see that $\sim^\prime\subseteq\sim$. Since $X_t/\sim^\prime$ is Hausdorff, it follows that $\sim^\prime$ is closed. Thus $\sim^\prime=\sim$ and $\sim\cap (X_t^\circ\times X_t^\circ)=\widehat{\sim}$.
\end{remark}
Note that by Definition \ref{dfn:sim}, there exists a map $h_t:Y_t \to M$ such $F_\psi|_{X_t}=h_t\circ q_t$. Since $q_t$, $F_\psi|_{X_t}$ are holomorphic and harmonic, respectively, it follows that $h_t$ is harmonic.
\subsection{Semicontinuity of $Y_t$}\label{subsec:semicontinuity} Our main result here is the following.
\begin{lemma}\label{lm:semicontinuity}
    For any $t\in\mathcal{S}^\circ$, there exists a neighbourhood $\mathcal{U}\ni t$ in $\mathcal{S}^\circ$, such that for any $s\in \mathcal{U}$, there is a non-constant holomorphic map $\tau_{t,s}:Y_t\to Y_s$ with $h_t=h_s\circ\tau_{t,s}$.
\end{lemma}
\begin{proof}
    For any $z\in\Curve_{\mathcal{S}^\circ}$, let $U_z$ be a neighbourhood of $z$ such that there exists a holomorphic map $\phi_z:U_z\to\Lambda_z$, where $\Lambda_z$ is a Riemann surface homeomorphic to a disk, and the fibres of $\phi_z$ are connected open subsets of the leaves of $\mathcal{F}$. Note that $F_\psi$ is constant on the leaves of $\mathcal{F}$, so we can write $F_\psi|_{U_z}=\tilde{h}_z\circ\phi_z$, where $\tilde{h}_z:\Lambda_z\to M$.
    \par Consider the maps $\theta_z:=(\pi, \phi_z):U_z\to \mathcal{S}\times \Lambda_z$. Note that any fibre of $\theta_z$ is an analytic subset of $X_s\cap U_z$ for some $s\in\mathcal{S}^\circ$, and is hence discrete (in $X_s\cap U_z$). By \cite[Open mapping theorem, pp. 107]{Grauert1984}, the map $\theta_z$ is open. We can thus shrink $U_z\ni z, \Lambda_z$ if necessary, such that the image of $\theta_z$ takes the form $\mathcal{U}_z\times\Lambda_z$, for $\mathcal{U}_z\subset\mathcal{S}^\circ$. This in particular implies that $\phi_z(U_z\cap X_s)=\Lambda_z$ for all $s\in\pi(U_z)$.
    \par Now fix an arbitrary $t\in\mathcal{S}^\circ$, as in the statement of the lemma. Let $z_1,z_2,...,z_N\in X_t$ be such that $X_t\subset\bigcup_{i=1}^N U_{z_i}$, and set $U_i:=U_{z_i}, \phi_i:=\phi_{z_i}, \Lambda_i:=\Lambda_{z_i}, \tilde{h}_i:=\tilde{h}_{z_i}$. Let $\mathcal{U}:=\bigcap_{i=1}^N \pi(U_i)$, and observe that by construction $\mathcal{U}\subseteq\mathcal{S}^\circ$ is an open neighbourhood of $t$. For $s\in\mathcal{U}$, let $D_s\subset X_s$ be the finite set of points where $\phi_i|_{X_s}$ is not an immersion, for some $i=1,2,...,N$.
    \begin{claim}\label{claim:fancy-local-form}
        For any $s\in\mathcal{U}$ and $i\in\{1,2,...,N\}$, there exists a holomorphic map $\xi_{i,s}:\Lambda_i\to Y_s$ such that $q_s=\xi_{i, s}\circ\phi_i$  on $X_s\cap U_i$ and $h_s\circ\xi_{i,s}=\tilde{h}_i$ on $\Lambda_i$.
    \end{claim}
    \begin{proof}
         Suppose $a, b\in U_i\cap X_s\setminus D_s$ are such that $\phi_i(a)=\phi_i(b)$. Then since $\phi_i$ is a local biholomorphism near $a, b$, there exist neigbhourhoods $A\ni a, B\ni b$ in $X_s\cap U_i$, and a biholomorphism $\eta:A\to B$ such that $\phi\circ\eta=\phi|_A$. Recall that $F_\psi|_{U_i}=h_i\circ\phi_i$. Then $F_\psi\circ\eta=F_\psi|_A$, and hence $a\sim b$. Therefore $q_s(a)=q_s(b)$.
         \par It follows that there exists a holomorphic map $\xi_{i,s}:\phi_i(U_i\cap X_s\setminus D_s)\to Y_s$ such that $q_s|_{U_i\cap X_s\setminus D_s}=\xi_{i,s}\circ\phi_i|_{U_i\cap X_s\setminus D_s}$. Since $\phi_i(D_s\cap U_i)$ is finite and $\phi_i(U_i\cap X_s)=\Lambda_i$, the map $\xi_{i,s}$ extends to $\Lambda_i$ and the factorization $q_s=\xi_{i,s}\circ\phi_i$ extends over $D_s$ by continuity.
         \par Observe that $h_s\circ\xi_{i,s}\circ\phi_i=h_s\circ q_s=F_\psi=\tilde{h}_i\circ\phi_i$. Since $\phi_i(U_i\cap X_s)=\Lambda_i$, the second claim follows.
    \end{proof}
    From Claim \ref{claim:fancy-local-form}, it follows that we have the following commutative diagram for any $s\in\mathcal{U}$,
    \[
            \begin{tikzcd}
                X_t\cap U_i\arrow[d, "q_s"] \arrow[r, "\phi_i"] & \Lambda_i \arrow[d, "\tilde{h}_i"]\arrow[dr, "\xi_{i,s}"] & \arrow[l, "\phi_i"] X_s\cap U_i\arrow[d, "q_s"] \\ 
                Y_t \arrow[r, "h_t"] & M & \arrow[l, "h_s"] Y_s
            \end{tikzcd}    
            \]
    Now fix an arbitrary $s\in\mathcal{U}$. Denote $X_t^-=X_t^\circ\setminus D_t$ and define $f:X_t^-\to Y_s$ by $f|_{U_i\cap X_t^-}=\xi_{i,s}\circ\phi_i|_{U_i\cap X_t^-}$. Note that $X_t^-$ is $X_t$ with a finite number of points removed. 
    \begin{claim}
        The map $f$ is well-defined and constant on the fibres of $q_t$.
    \end{claim}
    \begin{proof}
        We prove both claims simultaneously. Suppose $a\in X_t^-\cap U_i,b\in X_t^-\cap U_j$ have neighbourhoods $A\ni a, B\ni b$ in $X_t^-$, and a biholomorphism $\eta:A\to B$ such that $F_\psi\circ\eta=F_\psi|_A$. We will show that $\xi_{i,s}(\phi_i(a))=\xi_{j,s}(\phi_j(b))$. Note that   
        \begin{enumerate}
            \item when $a=b$ and $\eta=\mathrm{id}$, this shows that $f$ is well-defined, and 
            \item when $a\sim b$, this shows that $f$ is constant on the fibres of $q_t$.
        \end{enumerate}
        \par Since $a, b\not\in D_t$, the corresponding maps $\phi_i:A\to\Lambda_i$ and $\phi_j:B\to\Lambda_j$ are local biholomorphisms, where we shrink $A,B$ without mention as necessary. Thus there exists a holomorphic map $\bar{\eta}:\phi_i(A)\to\phi_j(B)$ such that $\bar{\eta}\circ\phi_i=\phi_j\circ\eta$. In particular, $\bar{\eta}(\phi_i(a))=\phi_j(b)$.
        \par We analyze the diagram 
        \[
        \begin{tikzcd}
            A\arrow[r, "\phi_i"]\arrow[dd, "\eta"] & \phi_i(A)\arrow[dd, "\bar{\eta}"] \subseteq\Lambda_i\arrow[dr, "\xi_{i,s}"] & & \\
                                                   &                                                                        & Y_s\arrow[r, "h_s"] & M \\ 
            B\arrow[r, "\phi_j"] & \phi_j(B) \subseteq\Lambda_j\arrow[ur, "\xi_{j,s}"]
        \end{tikzcd}
        \]
        Observe that 
        \begin{align*}
            \tilde{h}_j\circ\bar{\eta}\circ\phi_i=F_\psi\circ\eta=F_\psi=\tilde{h}_i\circ\phi_i
        \end{align*} 
        on $A$. Thus, since $\phi_i$ is a local biholomorphism, we have $\tilde{h}_j\circ\bar{\eta}=\tilde{h}_i|_{\phi_i(A)}$. Therefore $h_s\circ\xi_{j,s}\circ\bar{\eta}=h_s\circ\xi_{i,s}$ on $\phi_i(A)\subseteq\Lambda_i$. 
        \par We now let $u_n\in \phi_i(A), v_n\in \phi_j(B)$ be sequences of points such that 
        \begin{enumerate}
            \item $u_n\to \phi_i(a), v_n\to \phi_j(b)$ as $n\to\infty$,
            \item neither of $\xi_{i,s}(u_n), \xi_{j,s}(v_n)$ is the $q_s$-image of a ramification point of $q_s:X_s\to Y_s$, and 
            \item $\xi_{i,s}$ (resp. $\xi_{j,s}$) is an immersion near $u_n$ (resp. $v_n$).
        \end{enumerate}
        Observe that there exist neigbhourhoods $U_n\ni \xi_{i,s}(u_n), V_n\ni \xi_{j,s}(v_n)$ and a biholomorphism $\eta_n:U_n\to V_n$ given by $\eta_n\circ\xi_{i,s}=\xi_{j,s}\circ\bar{\eta}$, such that $h_s\circ\eta_n=h_s$. By construction of $Y_s$, it follows that $\eta_n$ is the identity and $U_n=V_n$. In particular, $\xi_{i,s}(u_n)=\xi_{j,s}(v_n)$, and by taking limits we see that $\xi_{i,s}(\phi_i(a))=\xi_{j,s}(\phi_j(b))$, as desired.
    \end{proof}
    We now finish the proof of Lemma \ref{lm:semicontinuity}. Note that $f$ extends to a holomorphic map $f:X_t\to Y_s$ that is also constant on the fibres of $q_t$. Therefore $f=\tau_{t,s}\circ q_t$, where $\tau_{t,s}:Y_t\to Y_s$ is the desired holomorphic map. Note that $h_s\circ\tau_{t,s}\circ q_t=h_s\circ\xi_i\circ\phi_i=\tilde{h}_i\circ\phi_i=F_\psi=h_t\circ q_t$, so since $q_t:X_t\to Y_t$ is surjective, we see that $h_s\circ\tau_{t,s}=h_t$ as claimed.
\end{proof}
\subsection{Finishing the proof}\label{subsec:finishing-thm-inframain}
Denote by $g(t)$ the genus of $Y_t$ for $t\in\mathcal{S}^\circ$.  Since $\tau_{t,s}$ is a non-constant holomorphic map $Y_t\to Y_s$ for any $s$ in some neighbourhood $\mathcal{U}$ of $t$, it follows that $g(t)\geq g(s)$. Therefore 
\begin{align*}
g(t)\geq\limsup_{s\to t} g(s).
\end{align*}
Hence $g:\mathcal{S}^\circ\to\mathbb{Z}_{\geq 0}$ is an upper semicontinuous map. Let $g_0$ be the minimum of $g$. Thus there is a connected open subset $\mathcal{U}\subset\mathcal{S}^\circ$ such that $g(t)=g_0$ for $t\in\mathcal{U}$. Now fix $t\in\mathcal{U}$. After shrinking $\mathcal{U}$, we can assume that for all $s\in\mathcal{U}$, the map $\tau_{t,s}:Y_t\to Y_s$ is a holomorphic map between closed Riemann surfaces of the same genus, so by Riemann--Hurwitz it is an isomorphism.
\par Note now that for any $s\in\mathcal{U}$, the Riemann surface $X_s$ admits a non-constant holomorphic map $\tau_{t,s}^{-1}\circ q_s:X_s\to Y:=Y_t$. Hence $\mathcal{U}$ is covered by $\bigcup_{[p]\in[\Sigma_g, Y]}\Lambda_{Y,p}$. Here $[\Sigma_g, Y]$ denotes the set of free homotopy classes of maps $\Sigma_g\to Y$. Thus, by the Baire category theorem \cite[Theorem 34, pp. 200]{Kelley2005}, for some homotopy class $p:\Sigma_g\to Y$, the set $\mathcal{U}\subset\Lambda_{Y,p}$, after possibly shrinking $\mathcal{U}$ to a smaller open set. Since $\mathcal{S}$ is connected and $\Lambda_{Y,p}$ is closed, it follows that $\mathcal{S}\subseteq\Lambda_{Y,p}$.
\par Finally, fix an arbitrary $t\in\mathcal{U}$, and observe that since $q_t\simeq p$, it follows that $\psi\simeq h_t\circ p$.
\section{Factorization results}\label{sec:factorization-results}
This section is devoted to the proof of Theorem \ref{thm:ph-smooth-projective}, and is organized as follows. In \S\ref{subsec:siu-sampson}, we pass from $X$ to its resolution of sinuglarities $X'$ using Hironaka's results \cite{Hironaka1964,Hironaka1964a}, and apply the Siu--Sampson theorem \cite{Siu1980,Sampson1986} to the lift of $f$ to $X'$. In \S\ref{subsec:bertini} we construct a diagram of holomorphic maps
\[
\begin{tikzcd}
    E \arrow[r,"p"]\arrow[d, "\phi"] &  B\\ 
    X'
\end{tikzcd}    
\]
where the fibres of $p$ are closed connected Riemann surfaces, and  $\phi$ is a proper dominant holomorphic map. Then by the conclusion of the Siu--Sampson theorem, and Proposition \ref{prop:energy-constant-iff-ph} show that the energy of the lift of $f$ over $p^{-1}(b)$ is constant as a function of $b\in B$. In \S\ref{subsec:applying-thm-supermain}, we apply Theorem \ref{thm:supermain} to finish the proof of Theorem \ref{thm:ph-smooth-projective}. In \S\ref{subsec:pf-claim-cd-props} and \S\ref{subsec:claim-hatphi-rank-one} we show two technical claims made in \S\ref{subsec:applying-thm-supermain}.
\subsection{Siu--Sampson argument}\label{subsec:siu-sampson}
Let $f:X\to M$ be as in Theorem \ref{thm:ph-smooth-projective}. Let $\rho:X'\to X$ be a bimeromorphic map, where $X'$ is a smooth projective variety. Such a map exists by Hironaka's results on resolution of singularities \cite{Hironaka1964,Hironaka1964a} (the reader can also consult \cite[Corollary 3.22]{Kollar2011}). Since $\phi$ is projective, it is also a K\"ahler morphism \cite{Bingener1983}, and hence $X'$ admits the structure of a compact K\"ahler manifold as in \cite[Proposition 1.3.1 (vi)]{Varouchas1989}. Moreover, for this K\"ahler structure, the map $f\circ\rho$ is harmonic on $X'\setminus\rho^{-1}(X^\mathrm{sing.})$. Here $X^\mathrm{sing.}$ is the singular locus of $X$.
\par Note that the real codimension of $\rho^{-1}(X^\mathrm{sing.})$ in $X'$ is at least two, so the relative 2-capacity vanishes: \[\mathrm{cap}_2\left(\rho^{-1}(X^\mathrm{sing.})\right)=0.\]
By a result of Meier \cite[Theorem 1]{Meier1986}, it follows that $f\circ\rho$ extends to a harmonic map $f':X'\to M$. Applying Sampson's theorem \cite{Sampson1978} (see also \cite{Loustau2020}), we see that $f'$ is pluriharmonic and 
\begin{align*}
    \mathrm{rank}_\mathbb{C}(Df'(T^{1,0}X'))\leq 1.
\end{align*}
\subsection{Constructing a cover by curves}\label{subsec:bertini}
Suppose that $X'\subseteq\mathbb{P}^N$ is a complex submanifold of $N$-dimensional projective space $\mathbb{P}^N$. Set $d=\dim_\mathbb{C}X'$.
\par We denote by $\mathrm{Gr}(k, n)$ the Grassmanian variety parameterizing $k$-dimensonal subspaces of $\mathbb{C}^n$. Recall that the tautological bundle \[\mathcal{O}_{\mathbb{P}^N}(-1)\to\mathbb{P}^N\] is a subbundle of the trivial $\mathbb{C}^{N+1}$-bundle over $\mathbb{P}^N$, such that the fibre of $\mathcal{O}_{\mathbb{P}^N}(-1)$ over a line $L\in\mathbb{P}^N$ is precisely $L\leq\mathbb{C}^{N+1}$. We denote by $\mathcal{O}_{X'}(-1)$ the restriction of $\mathcal{O}_{\mathbb{P}^N}(-1)$ to $X'$.
\par Define
\begin{align*}
    E=\{(x, V)\in X'\times\mathrm{Gr}(N-d+2, N+1): x\in V\}.
\end{align*} 
Set $B=\mathrm{Gr}(N-d+2, N+1)$, and let $\phi:E\to X'$ be the projection to the first factor, and $p:E\to B$ be the projection to the second factor.
\begin{claim}\label{claim:gcc}
    The diagram 
    \[\begin{tikzcd}
        E\arrow[r, "p"]\arrow[d, "\phi"] &  B\\
        X'
    \end{tikzcd}\]
    has the following properties 
    \begin{enumerate}
        \item $E$ is a smooth projective variety and $\phi$ is a proper, dominant morphism,  
        \item\label{item:nice-fibres} a general fibre of $p$ is a smooth, connected, closed Riemann surface, and 
        \item\label{item:equiv-relation} the closed equivalence relation on $X'$ generated by 
        \begin{align*}
            \{(x_1, x_2)\in X':\text{there exists }b\in B\text{ with }\phi^{-1}(x_i)\cap p^{-1}(b)\neq\emptyset\text{ for }i=1,2\}
        \end{align*}
        has a single equivalence class.
    \end{enumerate}
\end{claim} 
\begin{proof}
    We show the different claims in turn.
    \begin{enumerate}
        \item Note that $\phi:E\to X$ is the Grassman $(N-d+1)$-plane bundle of the vector bundle $\mathbb{C}^{N+1}/\mathcal{O}_X(-1)$, and in particular $E$ is a smooth projective variety and $\phi:E\to X$ is proper and dominant.
        \item  By \cite[Theorem 1.22, pp. 108]{Grauert1994}, the generic fibre of $p$ is smooth. Moreover, since $\phi:E\to X$ is the Grassman $(N-d+1)$-plane bundle of the vector bundle $\mathbb{C}^{N+1}/\mathcal{O}_X(-1)$, we have 
        \begin{align*}
            \dim E=d+\dim\mathrm{Gr}(N-d+1, N)=d+(N-d+1)(d-1),
        \end{align*}
        and hence the general fibre of $p:E\to B$ has dimension $\dim E-\dim B=\dim E-(N-d+2)(d-1)=1$. 
        \par By \cite[Theorem 1.1]{Fulton1981}, the map $p:E\to B$ has a non-empty connected fibre over a general point.
        \item Immediate from the fact that any two points in $X'$ lie on some $(N-d+2)$-dimensional subspace.
    \end{enumerate} 
\end{proof} 
\subsection{Applying Theorem \ref{thm:supermain}}\label{subsec:applying-thm-supermain} Let $A\subset B$ be the exceptional subset to Claim \ref{claim:gcc}(\ref{item:nice-fibres}). Thus for any $b\in B\setminus A$, the fibre $p^{-1}(b)$ is a smooth connected closed Riemann surface.
\par Let $q:\tilde{B}\to B\setminus A$ be the universal cover of $B\setminus A$, and suppose that 
\[
\begin{tikzcd}
    q^*E\arrow[r, "q^*p"]\arrow[d, "q"] & \tilde{B} \arrow[d, "q"] \\
    E\setminus p^{-1}(A)\arrow[r, "p"] & B\setminus A
\end{tikzcd}    
\]
is a pullback diagram of the bundle $p:E\setminus p^{-1}(A)\to B\setminus A$. Note that $q^*p$ is topologically a product, since $\tilde{B}$ is simply connected. By Proposition \ref{prop:energy-constant-iff-ph} applied to the map $F\circ q:q^*E\to M$, where $F=f'\circ\phi$, we see that the energy of $(F\circ q)_b$ does not depend on $b\in\tilde{B}$. To apply Theorem \ref{thm:supermain}, it remains to show that $F\circ q:q^*E\to M$ is transverse to the fibres of $q^*p$. Assume to the contrary, then $F\circ q$ is constant on each fibre of $q^*E$. Thus $F=f'\circ\phi$ is constant on the fibres of $p:E\setminus p^{-1}(A)\to B\setminus A$. By Claim \ref{claim:gcc}(\ref{item:equiv-relation}), it follows that $f'$ is constant. Hence $f$ is constant, and there is nothing to prove.
\par Thus by Theorem \ref{thm:supermain}, $F\circ q$ either has image contained in a closed geodesic, or factors as $\tilde{h}\circ\tilde{\phi}$, where $\tilde{\phi}:q^*E\to Y$ is a holomorphic map to a closed hyperbolic Riemann surface $Y$ and $\tilde{h}:Y\to M$ is harmonic. If the image of $F\circ q$ is contained in a geodesic, so is the image of $f$, and there is nothing to prove. 
\par Assume therefore that we are in the latter case, $F\circ q=\tilde{h}\circ\tilde{\phi}$. Recall that $\phi:E\to X'$ is the holomorphic map from Claim \ref{claim:gcc}. Thus we have a commuting diagram 
\[
    \begin{tikzcd}
        q^*E\arrow[d, "\tilde{\phi}"]\arrow[r, "q"] & E\setminus p^{-1}(A)\arrow[r, "\phi"]  & X'\arrow[d, "f'"] \\ 
        Y \arrow[rr, "\tilde{h}"] &  & M
    \end{tikzcd}     
\]
Recall that $\pi_1(B\setminus A)$ acts on $\tilde{B}$ by covering transformations. This induces an action of $\pi_1(B\setminus A)$ on $q^*E$. 
\begin{claim}\label{claim:cd-props}
    The diagram above has the following properties:
    \begin{enumerate}
        \item\label{item:finite-index} There exists a finite index subgroup $\Gamma\leq\pi_1(B\setminus A)$ such that $\tilde{\phi}$ is invariant under $\Gamma$, and 
        \item\label{item:constant-conn-components} $\tilde{\phi}$ is constant on the connected components of any general fibre of $\phi\circ q$. 
    \end{enumerate}
\end{claim}
We defer the proof of Claim \ref{claim:cd-props} to the next subsection.
\par Let $m=[\pi_1(B\setminus A):\Gamma]$ be the index of $\Gamma$. We define tne a map $\hat{\phi}:E\setminus p^{-1}(A)\to \Sym^{m} Y$ by 
\begin{align*}
    \hat{\phi}(x)=\{\tilde{\phi}(\gamma\tilde{x}): \gamma\Gamma\in\mathrm{Deck}(q)/\Gamma\},
\end{align*}
where $\tilde{x}\in q^*E$ is some arbitrary lift of $x\in E\setminus p^{-1}(A)$. The right-hand side of this equation denotes a multiset, and $\Sym^d Y$ denotes the configuration space of $d$ (unlabelled) points on the Riemann surface $Y$. It is easy to see that $\hat{\phi}$ is holomorphic. 
\begin{claim}\label{claim:hatphi-rank-one}
    We have $\mathrm{rank}_\mathbb{C}D\hat{\phi}=1$ on an open dense subset of $E$.
\end{claim}
We defer the proof of Claim \ref{claim:hatphi-rank-one} until \S\ref{subsec:claim-hatphi-rank-one}.
\par Hence the image of $\hat{\phi}$ is a (complex) one-dimensional constructible subset $D\subset\Sym^m Y$, by Chevalley's theorem. Let $A'\subset D$ be such that $D\setminus A$ is a qusiprojective variety of complex dimension one. 
\par Since $E\setminus (p^{-1}(A)\cup \hat{\phi}^{-1}(A'))$ is a complex manifold, by standard properties of normalization \cite[Proposition, pp. 180]{Grauert1984}, $\hat{\phi}$ factors as 
\begin{align*}
    E\setminus (p^{-1}(A)\cup \hat{\phi}^{-1}(A'))\stackrel{\phi}{\longrightarrow} \hat{D} \stackrel{\theta}{\longrightarrow} D,
\end{align*}
where $\theta:\hat{D}\to D$ is the normalization of $D$. Normal one-dimensional complex analytic spaces are smooth, so $\hat{D}$ is in fact a finite type Riemann surface. We fill in the punctures of $\hat{D}$, and denote the resulting closed Riemann surface $\hat{D}$ as well, by a slight abuse of notation. Note that by the Schwarz lemma, the map $\phi:E\setminus (p^{-1}(A)\cup\hat{\phi}^{-1}(A'))\to \hat{D}$ is meromorphic in the sense of Andreotti \cite{Andreotti1960}, and can hence be extended to a holomorphic map $\theta:E\to\hat{D}$ by \cite[Theorem 4]{Andreotti1960}. Moreover, it is clear that $F$ factors through $\theta$. 
\par By Claim \ref{claim:cd-props}(\ref{item:constant-conn-components}), it follows that $\theta$ is constant on the connected components of the fibres of $\phi$. We let $\phi\circ \rho=\alpha\circ\beta$ be the Stein factorization of $\phi\circ\rho:E\to X$, such that $\beta:E\to S$ has connected fibres, and such that $\alpha:S\to X$ is finite. Let $K\subset X$ be the proper analytic subset containing the singular locus of $X$, such that $\alpha^{-1}(x)$ has exactly $\deg\alpha$ distinct points, for $x\in X\setminus K$. Thus we define $\hat{\theta}:X\setminus K\to\Sym^{\deg\alpha} \hat{D}$ by 
\begin{align*}
    \hat{\theta}(x)=\{\theta(y):y\in \alpha^{-1}(x)\}.
\end{align*}
From an argument that is identical to the proof of Claim \ref{claim:hatphi-rank-one} and the paragraph following its statement, it follows that there is a holomorphic map $\xi:X\to C$ for a closed Riemann surface $C$, such that $f$ factors through $\xi$.
\subsection{Proof of Claim \ref{claim:cd-props}}\label{subsec:pf-claim-cd-props}
\begin{enumerate}
    \item Fix $b\in B\setminus A$ and one of its $q$-preimages $\tilde{b}\in \tilde{B}$. Since $Y$ is a closed hyperbolic Riemann surface, there are only finitely many holomorphic maps $p^{-1}(b)\to Y$. Thus there is a finite index subgroup $\Gamma\leq\mathrm{Deck}(q)$ such that $\tilde{\phi}\circ\gamma=\tilde{\phi}$ on $(q^*p)^{-1}(\tilde{b})$. By uniqueness of harmonic maps in a given homotopy class, if $\tilde{\phi}\circ\gamma=\tilde{\phi}$ on one fibre of $q^*E$, then $\tilde{\phi}\circ\gamma=\tilde{\phi}$ on all of $q^*E$.  Thus $\tilde{\phi}\circ\gamma=\tilde{\phi}$ for all $\gamma\in\Gamma$.
    \item  Let $t\in X'$ be arbitrary such that $\phi^{-1}(t)$ is a closed complex submanifold (which is a Zariski open condition). Thus $T:=q^{-1}(\phi^{-1}(t)\setminus p^{-1}(A))$ is a complex submanifold, on which $\tilde{h}\circ\tilde{\phi}$ is constant by definition. However, $\tilde{\phi}(T)$ either has non-empty interior, or is a point (see \cite[Open mapping theorem for holomorphic functions, pp. 109]{Grauert1984}). If $\tilde{\phi}(T)$ has non-empty interior, then $\tilde{h}$ is constant on an open set in $Y$, and is hence constant on $Y$ by classical results on harmonic maps \cite{aronszajn1956unique}. Thus $\tilde{\phi}(T)$ is a singleton, and the result is shown.
\end{enumerate}
\subsection{Proof of Claim \ref{claim:hatphi-rank-one}}\label{subsec:claim-hatphi-rank-one}  Denote by $\mathrm{Deck}(q)\cong\pi_1(B\setminus A)$ the deck group of $q$.
\par  Let $\tilde{x}\in q^*E$ be a point such that $\tilde{\phi}$ has surjective derivative in a nieghbourhood $U$ of $\tilde{x}$ and such that $\tilde{h}$ is locally an immersion near $\tilde{\phi}(\gamma\tilde{x})$ for any $\gamma\in\mathrm{Deck}(q)$. Such points lie in comeagre subset of $q^*E$. 
    \par Let $U\ni\tilde{x}$ be a small neighbourhood of $\tilde{x}$ such that 
    \begin{enumerate}
        \item  $\tilde{\phi}$ is a holomorphic submersion over $\bigcup_{\gamma\in\mathrm{Deck}(q)}\gamma U$, and 
        \item $\tilde{h}$ is a local immersion on $\bigcup_{\gamma\Gamma\in\mathrm{Deck}(q)/\Gamma} \tilde{\phi}(\gamma U)$,
    \end{enumerate}
   which exists since $[\mathrm{Deck}(q):\Gamma]<\infty$ and $\tilde{\phi}$ is $\Gamma$-invariant. 
    \par Let $\mathcal{F}$ be the foliation on $\mathrm{Deck}(q)U$ whose leaves are connected components of the fibres of $\tilde{\phi}$. Since $\tilde{h}$ is locally an immersion, the leaves of $\mathcal{F}$ are precisely the connected components of the fibres of $F\circ q$. In particular, $\mathcal{F}$ is $\mathrm{Deck}(q)$-invariant. Let $\Lambda$ be the Riemann surface parameterizing the leaves of $\mathcal{F}$. 
    \par Then $\tilde{\phi}=\alpha\circ\beta$, where $\beta:\mathrm{Deck}(q)U\to \Lambda$ is the natural projection, and $\alpha:\Lambda\to Y$ is some holomorphic map. In particular, since $\beta$ is $\mathrm{Deck}(q)$-equivariant, it follows that 
    \begin{align*}
        \hat{\phi}=\alpha\left(\gamma\cdot\beta(x):\gamma\Gamma\in\mathrm{Deck}(q)/\Gamma\right).
    \end{align*}
    It follows that the image of $\hat{\phi}$ lies within the $\alpha$ image of the map $z\to \{\gamma z:\gamma\Gamma\in\mathrm{Deck}(q)/\Gamma\}$ defined on $\Lambda$. In particular, since $\dim_\mathbb{C}\Lambda=1$, it follows that $\mathrm{rank}_\mathbb{C}D\hat{\phi}\leq 1$.
    \par Since $\tilde{\phi}$ is not constant, neither is $\hat{\phi}$, and hence $\mathrm{rank}_\mathbb{C}D\hat{\phi}\geq 1$ at a general point, concluding the proof of the claim.
\section{Mapping class groups}\label{sec:mcg}
In this section, we show Theorem \ref{thm:mcg}, following the argument of \cite{Markovic2024}. In \S\ref{subsec:bridson}, we recall the result of Bridson \cite{Bridson2010} that is used in the proof of Theorem \ref{thm:mcg}, that we then show in \S\ref{subsec:pf-mcg}.
\subsection{A result of Bridson}\label{subsec:bridson}
We first state some preliminary results due to Bridson \cite[Theorem B, Remark 1]{Bridson2010} that will be crucial in our proof of Theorem \ref{thm:mcg}.
\begin{proposition}
    Let $\Gamma\leq\Mod_{g,n}$ be a finite index subgroup of the mapping class group of a surface of genus $g\geq 3$ with $n\geq 0$ punctures. Let $X$ be a $\mathrm{CAT}(0)$ space and $\phi:\Gamma\to\mathrm{Isom}(X)$ be a homomorphism whose image consists of hyperbolic isometries. Then any power of a Dehn multitwist that lies in $\Gamma$ also lies in $\ker\phi$.
\end{proposition}
We will apply this with $X$ being the universal cover of $M$. Since $M$ is convex cocompact, the deck group of this cover $X\to M$ consists entirely of hyperbolic isometries. Using the fact that any point-pushing mapping class that corresponds to a simple closed curve is a mutlitwist, we get the following corollary.
\begin{corollary}\label{cor:bridson}
    Let $M$ be as in Assumptions \ref{ass:mfd}, and let $\Gamma\leq\Mod_{g,n+1}$ be a finite index subgroup, with $g\geq 3, n\geq 0$. Then for any homomorphism $\phi:\Gamma\to\pi_1(M)$ and for any simple closed curve $\gamma$ on $\Sigma_{g,n}$, we have $\phi(\iota(\gamma)^k)=0$, whenever $\iota(\gamma)^k\in \Gamma$.
\end{corollary}
Recall that $\iota$ is the map that embeds $\pi_1(\Sigma_{g,n})$ as the point-pushing subgroup $\Pi_{g,n}\leq\Mod_{g,n+1}$.
\subsection{Proof of Theorem \ref{thm:mcg}}\label{subsec:pf-mcg}
Let $M$ be a manifold satisfying Assumptions \ref{ass:mfd}, $\Gamma\leq\Mod_{g,n+1}$ be a finite index subgroup and $\phi:\Gamma\to\pi_1(M)$ be a strongly point-pushing homomorphism. Recall the Birman exact sequence 
\begin{align*}
    1\to\pi_1(\Sigma_{g,n})\stackrel{\iota}{\longrightarrow}\Mod_{g,n+1}\stackrel{\mathcal{F}}{\longrightarrow}\Mod_{g,n}\to 1,
\end{align*}
and the point-pushing subgroup $\Pi_{g,n}=\mathrm{im}(\iota)\leq\Mod_{g,n+1}$.
\par Let $K=\iota^{-1}(\Pi_{g,n}\cap \Gamma)\leq\pi_1(\Sigma_{g,n})$. Then, as $[\Mod_{g,n+1}:\Gamma]<\infty$, we have $[\pi_1(\Sigma_{g,n}):K]<\infty$. Thus $K\leq\pi_1(\Sigma_{g,n})$ corresponds to a finite covering map $p:\Sigma_{h,m}\to\Sigma_{g,n}$ via $\mathrm{im}(p_*)=K$. Note that by Corollary \ref{cor:bridson}, the map $\phi\circ\iota\circ p_*$ annihilates any simple closed curve on $\Sigma_{h,m}$. In particular, it descends to a map $\pi_1(\Sigma_h)\to\pi_1(M)$. Let $\psi:\Sigma_h\to M$ be a continuous map representing $\phi\circ\iota\circ p_*$, which exists since both $\Sigma_h$ and $M$ are Eilenberg--MacLane spaces for their respective fundamental groups. 
\par  We denote by $\sigma_p$ the holomorphic map $\Teich_{g,n}\to\Teich_{h}$ obtained by lifting a complex structure $X\in\Teich_{g,n}$ via $p$ to $\Sigma_{h,m}$, and then filling in the punctures of the resulting Riemann surface. Note that Theorem \ref{thm:mcg}(\ref{item:geometry}) follows immediately from Theorem \ref{thm:inframain} once we show that $\mathrm{E}_\psi$ is constant on the image of $\sigma_p$. We now show how to conclude Theorem \ref{thm:mcg}(\ref{item:topology}) from the constancy of $\mathrm{E}_\psi\circ\sigma_p$, and then we show that $\mathrm{E}_\psi\circ\sigma_p$ is constant in the next subsection. 
\par Assume that $\mathrm{E}_\psi\circ\sigma_p$ is constant. By Theorem \ref{thm:inframain} (and Proposition \ref{prop:schiffer}), the map $F_\psi$ factors over $\sigma_p(\Teich_{g,n})$ through a holomorphic map $\xi:\Teich_{g,n+1}/\Gamma\cap\Pi_{g,n}\to Y$ to a closed hyperbolic Riemann surface $Y$. By \cite[Propositions 2.3 and 2.4]{Markovic2022}, the map $\xi$ is invariant under a finite index subgroup $\Gamma'\leq\Mod_{g,n+1}$. This shows Theorem \ref{thm:mcg}(\ref{item:topology}) with $\Theta=\Gamma\cap\Gamma'$.
\subsubsection{The energy $\mathrm{E}_\psi$ is constant on the image of $\sigma_p$} \label{subsubsec:energy-constant}
We first observe that $\mathrm{E}_\psi\circ\sigma_p$ is invariant under the action of the finite index subgroup $\mathcal{F}(\Gamma)\leq\Mod_{g,n}$. This follows from uniqueness of harmonic maps, and the fact that the free homotopy class of $\psi$ is invariant under the action of $\Gamma$. 
\begin{claim}\label{claim:bounded}
    The energy $\mathrm{E}_\psi\circ\sigma_p$ is bounded.
\end{claim}
We first show how to prove that $\mathrm{E}_\psi\circ\sigma_p$ is constant assuming Claim \ref{claim:bounded}. By the result of Boggi--Pikaart \cite[Corollary 2.10]{Boggi2000}, there exists a finite index subgroup $\bar{\Gamma}\leq\mathcal{F}(\Gamma)$, such that $\Teich_{g,n}/\bar{\Gamma}$ has a compactification $\mathcal{M}$ which is a smooth projective variety. Then by standard theory in complex analysis \cite[Theorem (5.24)]{Demailly2012}, the function $\mathrm{E}_\psi\circ\sigma_p:\Teich_{g,n}/\bar{\Gamma}\to\mathbb{R}$ extends to a bounded plurisubharmonic function $\mathcal{M}\to\mathbb{R}$. Then by the strong maximum principle, $\mathrm{E}_\psi\circ\sigma_p$ is constant.
\begin{proof}[Proof of Claim \ref{claim:bounded}]
    The proof is nearly identical to that of \cite[Proposition 6.4]{Markovic2024}, so we only give a sketch.
    \par Suppose that $X_1,X_2,...\in\Teich_{g,n}$ has $\mathrm{E}_\psi(\sigma_p(X_n))\to\infty$. Since $\mathcal{F}(\Gamma)\leq\Mod_{g,n}$ is finite index, there exist mapping classes $T_n\in\mathcal{F}(\Gamma)$ such that $T_nX_n\to Y$, where $Y$ is a marked noded Riemann surface. 
    \par Let $Z$ be the marked noded Riemann surface obtained by lifting the complex structure of $Y$ via $p$. Let $\gamma_1,\gamma_2,...,\gamma_k\in\pi_1(\Sigma_h)$ be the (disjoint) simple closed curves that correspond to nodes of $Z$. Then by Corollary \ref{cor:bridson}, we have $\psi_*(\gamma_i)=1$ for $1\leq i\leq k$. 
    \par After applying a suitable homotopy, we may assume without loss of generality that $\psi:Z\to M$ is smooth on $Z$ and constant in a neighbourhood of each node. But then \[\mathrm{E}_\psi\circ\sigma_p(X_n)=\mathrm{E}_\psi\circ\sigma_p(T_nX_n)\to \mathrm{E}_\psi(Z)\leq \int_Z \abs{D\psi}^2 d\mathrm{vol}_Z<\infty,\]
    leading to the desired contradiction.
\end{proof}
\appendix
\section{Non-abelian Hodge correspondence over a moving Riemann surface}\label{appendix}
The goal of this section is to prove Theorem \ref{thm:nah}, i.e. the non-abelian Hodge correspondence for $\mathrm{PSL}(2,\mathbb{R})$ over a moving Riemann surface. If $\chi_{g,k}$ denotes the representation variety $\pi_1(\Sigma_g)\to\mathrm{PSL}(2,\mathbb{R})$ consisting of representations of Euler class $k$, then Theorem \ref{thm:nah} identifies $\Teich_g\times\chi_{g,k}$ via the non-abelian Hodge correspondence to the total space of a holomorphic fibration $\mathcal{E}(g,k)\to \Teich_g$. For $k>0$, the fibres paramaterize the divisor of the holomorphic energy, and the Hopf differential of the associated equivariant harmonic map to the hyperbolic plane $\mathbb{H}$.
\par An analogous statement was shown by Hitchin over a fixed Riemann surface \cite[Theorem (10.8)]{Hitchin1987}. Over a moving Riemann surface, the moduli spaces of solutions to the Hitchin equations for $\mathrm{GL}(2,\mathbb{C})$ were constructed by the author in the previous paper \cite[\S 4]{Tosic2023}. Since only representations of even Euler class can be lifted to $\mathrm{SL}(2,\mathbb{R})$, here we show Theorem \ref{thm:nah} by combining this construction from \cite{Tosic2023} with the arguments from \cite[\S 4]{Biswas2021}.
\par We first recall some notation and define the Euler class in \S\ref{subsec:appendix-notation}. Then we construct $\mathcal{E}(g,k)$ in \S\ref{subsec:constructing-egk}. We then show Theorem \ref{thm:nah} for $k$ even in \S\ref{subsubsec:even}, and for $k$ odd in \S\ref{subsubsec:odd}.
\subsection{Notation and preliminary bundless}\label{subsec:appendix-notation}
Recall that $\pi:\mathcal{V}_g\to\Teich_g$ is the universal curve, and that $\Sym^n\pi:\Sym^n\Curve_g\to\Teich_g$ is the holomorphic fibration over $\Teich_g$ whose fibre over $S\in\Teich_g$ is $\Sym^n S$.
\par Let $\chi_g$ be the space of representations $\pi_1(\Sigma_g)\to\mathrm{PSL}(2,\mathbb{R})$. For any representation $\rho\in\chi_g$, define its \textit{Euler class} $\mathrm{eu}(\rho)$ to be the Euler number of the flat $\mathbb{RP}^1$ bundle associated to $\rho$ (via the natural action of $\mathrm{PSL}(2,\mathbb{R})$ on $\mathbb{RP}^1$) \cite[\S 10, pp. 117]{Hitchin1987}. Then the Euler class completely classifies the connected components of $\chi_{g,k}$. 
\par We will make use of the fibred product 
\[
    \begin{tikzcd}
        \Sym^n\Curve_g\times_{\Teich_g}\Curve_g \arrow[r, "\rm{pr}_2"]\arrow[d, "\rm{pr}_1"] & \Curve_g \arrow[d, "\pi"]\\
        \Sym^n\Curve_g\arrow[r, "\Sym^n\pi"] & \Teich_g
    \end{tikzcd}
\]
Recall that the fibre of $\Sym^n\Curve_g\times_{\Teich_g}\Curve_g$ over a point $S\in\Teich_g$ is the product $\Sym^n S\times S$. Let $\mathcal{U}_n$ be the tautological line bundle over $\Sym^n\Curve_g\times_{\Teich_g}\Curve_g$ that restricts to $\mathcal{O}(D)$ over the section of $\Sym^n S\times S$ that corresponds $D\in\Sym^n S$, i.e. to the effective degree $n$ divisor $D$ on $S$. Let $K_\pi$ be the relative canonical bundle of $\pi$, i.e. the bundle over $\Curve_g$ that restricts to $K_S$ over the fibre of $\pi$ that corresponds to $S\in\Teich_g$.
\subsection{Constructing $\mathcal{E}(g,k)$}\label{subsec:constructing-egk}We construct $\mathcal{E}(g,k)$ as a holomorphic vector bundle over $\Sym^{2g-2-k}\Curve_g$.  Consider the bundle $\mathcal{U}_{2g-2-k}^{-1}\otimes \mathrm{pr}_2^*K_\pi^2\to\Sym^n\Curve_g\times_{\Teich_g}\Curve_g$. This is a holomorphic vector bundle that resticts over $\{D\}\times S\subset\Sym^{2g-2-k} S\times S=(\pi\circ\mathrm{pr}_2)^{-1}(S)$ to the line bundle $K_S^2(-D)$. Let $\mathcal{E}(g,k)$ be the pushforward of this bundle $\mathcal{U}_{2g-2-k}^{-1}\otimes \mathrm{pr}_2^*K_\pi^2$ via the map $\mathrm{pr}_1:\Sym^{2g-2-k}\Curve_g\times_{\Teich_g}\Curve_g\to\Sym^{2g-2-k}\Curve_g$. 
\par Thus the fibre of $\mathcal{E}(g,k)$ over $D\in\Sym^{2g-2-k} S$ for $S\in\Teich_g$, is naturally isomorphic to the sections of $K_S^2(-D)$. Here in the notation of \S\ref{subsec:nah}, $D$ represents the divisor of $\delta$, and the fibre $\mathcal{E}(g,k)_S=H^0(K_S^2(-D))$ represents the space of holomorphic quadratic differentials that vanish along $D$.
\subsection{Even Euler class}\label{subsubsec:even}
Here we show Theorem \ref{thm:nah} for even $k$. We will define a map $\mathrm{T}:\mathcal{E}(g,k)\to\Teich_g\times\chi_{g,k}$ in two steps: we first define $\mathrm{T}$ locally as a smooth diffeomorphism, then we show that the different local definitions agree. 
\par \textbf{Step 0. Local choices. } Therefore we fix a point $(X_0, D_0, \Phi_0)$ in (the total space of) $\mathcal{E}(g,k)$, that consists of 
\begin{enumerate}
    \item a marked Riemann surface $X_0\in\Teich_g$, 
    \item an effective degree $2g-2-k$ divisor $D_0$ on $X_0$, and 
    \item a holomorphic section $\Phi_0\in H^0(K_{X_0}^2(-D_0))$.
\end{enumerate}
Let $U$ be a small neighbourhood of this point in $\mathcal{E}(g,k)$. We will freely shrink $U$ without mention throughout the proof if necessary. 
\par Let $\mathcal{U}_{2g-2-k}^{1/2}$ be a square root of $\mathcal{U}_{2g-2-k}$ over \begin{align*}
    V&=\mathrm{pr}_1^{-1}(\{(D, X): (X,D,\Phi)\in U\text{ for some }\Phi\})\\
    &=\{(D, x):(X, D, \Phi)\in U\text{ for some }\Phi\text{, and }x\in X\}\subset\Sym^{2g-2-k}\Curve_g\times_{\Teich_g}\Curve_g,
\end{align*}
which exists since $k$ is even. We similarly pick a consistent square root $K_X^{1/2}$ of $K_X$ over $V$. 
\par \textbf{Step 1. Definition of $\mathrm{T}$. } Given $(X, D, \Phi)\in U$, consider the Higgs bundle $E=L\oplus L^{-1}$ over $X$, where $L=\left.\mathcal{U}_{2g-2-k}^{-1/2}\right|_{\{D\}\times X}\otimes K_X^{1/2}$ with Higgs field 
\begin{align*}
    \phi=\begin{pmatrix}
        & \Phi\sigma^{-1} \\ \sigma & 
    \end{pmatrix},
\end{align*}
where $\sigma$ is an arbitrary section of $K_XL^{-2}\cong \mathcal{O}(D)$ with divisor $D$. Note that there is ambiguity here in the scaling of $\sigma$, but all such Higgs bundles are gauge equivalent by \cite[Proposition (10.2)]{Hitchin1987}. By \cite[Proposition (10.2)]{Hitchin1987}, $(E,\phi)$ is stable, the harmonic metric on $(E,\phi)$ is diagonal with respect to the splitting $E=L\oplus L^{-1}$, and the associated flat bundle has holonomy $\rho:\pi_1(\Sigma_g)\to\mathrm{SL}(2,\mathbb{R})$ which projects to $\chi_{g,k}$. We set $\mathrm{T}(X, D, \Phi)=(X, \rho)$. Note that once we show that $\mathrm{T}$ is a diffeomorphism, the description of the zero section follows immediately since $\mathrm{tr}\left(\phi^2\right)=2\Phi$ is the Hopf differential of the $\rho$-equivariant harmonic map (see \cite[\S 5.2]{Li2019})
\par \textbf{Step 2. $\mathrm{T}$ is well-defined. } We need to show that $\mathrm{T}$ does not depend on the local choices of square roots of $\mathcal{U}_{2g-2-k}$ and $K$. Fix a triple $(X,D,\Phi)\in U$, and let $L$ be as above. Let $\tilde{L}$ be a line bundle obtained as above with a different choice of square roots of $\mathcal{U}_{2g-2-k},K_X$, and let $\rho, \tilde{\rho}:\pi_1(\Sigma_g)\to\mathrm{SL}(2,\mathbb{R})$ be the respective representations. Then $\tilde{L}=L\xi$, where $\xi$ is a line bundle such that $\xi^{\otimes 2}\cong \mathcal{O}_X$. In particular the Higgs bundle with this choice $\tilde{L}$ is 
\begin{align*}
    \tilde{E}=\tilde{L}\oplus\tilde{L}^{-1}=\xi\otimes\left(L\oplus L^{-1}\right)=\xi\otimes E.
\end{align*}
In particular, the projective bundles $\mathbb{P}(E)$ and $\mathbb{P}(\tilde{E})$ coincide. Thus $(E,\phi)$ and $(\tilde{E},\tilde{\phi})$ agree as holomorphic $\mathrm{PSL}(2,\mathbb{C})$-Higgs bundles, so $q\circ\rho=q\circ\tilde{\rho}$, where $q:\mathrm{SL}(2,\mathbb{R})\to\mathrm{PSL}(2,\mathbb{R})$ is the natural quotient map.  
\par \textbf{Step 3. $\mathrm{T}$ is smooth. } It suffices to show that the harmonic metric on $L$ depends smoothly on $(X,D,\Phi)$. Note that $L$ is a line bundle over $V$. We can also choose $\sigma$ to be a holomorphic section of $KL^{-2}$ over $V$. Let $\ell_0$ be an arbitrary background Hermitian metric on $L$, and let $h\in C^\infty(T^*\Curve_g^{\otimes 2})$ that restricts to the hyperbolic metric on the fibres of $\pi:\Curve_g\to\Teich_g$.
\par The Hitchin equation for the diagonal metric on $E=L\oplus L^{-1}$ for $(E,\phi)$ over $X\in\Teich_g$ is 
\begin{align*}
    i\Lambda_\omega \Theta_\ell+\norm{\Phi\sigma^{-1}}_{\ell,h}^2-\norm{\sigma}^2_{\ell,h}=0,
\end{align*}
where $\omega$ is the volume form, and $\ell$ is the metric on $L$ \cite[\S 2.3, equation (2.5)]{Biswas2021}. If we write $\ell=e^{\phi}\ell_0$, then the equation becomes
\begin{align}\label{eq:hitchin-decomp}
    i\Lambda_\omega\Theta_{\ell_0}+\Delta_g \phi-e^{2\phi}\norm{\Phi\sigma^{-1}}_{\ell_0,h}^2+e^{-2\phi}\norm{\sigma}_{\ell_0,h}^2=0.
\end{align}
This equation is analogous to \cite[\S 2.3, equation (2.6)]{Biswas2021}. We know by the non-abelian Hodge theorem that (\ref{eq:hitchin-decomp}) admits a unique solution for any $(X,D,\Phi)\in U$.  The following claim shows that this solution depends smoothly on the data $\norm{\Phi\sigma^{-1}}_{\ell_0}^2$ and $\norm{\sigma}_{\ell_0}^2$, that in turn depend smoothly on $(X, D, \Phi)\in U$. Thus $\mathrm{T}$ is smooth. 
\begin{claim}\label{claim:smooth}
    Let $\Sigma$ be a closed surface. Suppose that $g_0$ is a Riemannian metric on $\Sigma$, and that $\lambda_0, \mu_0, \phi_0:\Sigma\to\mathbb{R}$ are smooth functions, such that at least one of $\lambda_0,\mu_0$ does not vanish identically. Then there exists a neighbourhood $U$ of $(g_0, \lambda_0, \mu_0)\in C^\infty(\Sym^2T^*\Sigma)\times C^\infty(\Sigma)\times C^\infty(\Sigma)$ such that, for any $(g, \lambda,\mu)\in U$, there exists a smooth function $\phi$ solving the equation 
    \begin{align}\label{eq:claim-smooth}
        \Delta_g\phi+\mu^2 e^{-2\phi}-\lambda^2 e^{2\phi}=\Delta_{g_0}\phi_0+\mu_0^2 e^{-2\phi_0}-\lambda_0^2e^{2\phi_0},
    \end{align}
    where $\Delta_g$ is the Laplacian with respect to the Riemannian metric $g$. Moreover, this solution depends smoothly on $(g,\lambda,\mu)$.
\end{claim}
\begin{proof}
    Consider the map $U\times C^\infty(\Sigma)\to C^\infty(\Sigma)$ given by 
    \begin{align*}
        (g,\lambda,\mu,\phi)\longrightarrow\Delta_g \phi+\mu^2 e^{-2\phi}-\lambda^2 e^{2\phi}.
    \end{align*}
    This is a smooth map with derivative in the $\phi$-direction given by 
    \begin{align*}
        F[\dot{\phi}]=\Delta_g\dot{\phi}-\left(\lambda^2 e^{2\phi}+\mu^2 e^{-2\phi}\right)\dot{\phi},
    \end{align*}
    which is an elliptic self-adjoint operator. Note that $F$ is injective by the strong maximum principle, and thus also surjective. Hence $F$ is an isomorphism, and the result follows by the implicit function theorem.
\end{proof}
\par \textbf{Step 4. $\mathrm{T}$ is a diffeomorphism. } The fact that $\mathrm{T}$ is smooth implies that $\mathrm{T}$ is a diffeomorphism from the implicit function theorem, in the version stated below. 
\begin{claim}
    Let $\mathrm{T}:P_1\to P_2$ be a smooth map between smooth fibrations $p_i:P_i\to X$ over a manifold $X$, i.e. a smooth map making the following diagram commute
    \[\begin{tikzcd}
        P_1 \arrow[rr, "\mathrm{T}"]\arrow[dr, "p_1"] && P_2\arrow[dl, "p_2"]\\
        & X &
    \end{tikzcd}\]
    Then if $\mathrm{T}:p_1^{-1}(x)\to p_2^{-1}(x)$ is a diffeomorphism for all $x\in X$, then $\mathrm{T}$ is a diffeomorphism between total spaces $P_1\to P_2$.
\end{claim}
\begin{proof}
    It suffices to show that $T$ is a local diffeomorphism since it is clearly a bijection. We pick local coordinates $(x, y)$ of $P_1$, $(x, z)$ of $P_2$, and $x$ on $X$, such that $p_1$ takes the form $(x,y)\to x$ and $p_2$ takes the form $(x,z)\to x$. Here $x\in\mathbb{R}^{\dim X}, y,z\in\mathbb{R}^{\dim P_i-\dim X}$. We represent $\mathrm{T}$ locally as a map $(x, y)\to (x, F(x, y))$. Then $F(x,-)$ is a local diffeomorphism, so in particular $\frac{\partial F}{\partial y}$ is an isomorphism. Therefore by the implicit function theorem, there is a smooth map $G(x, z)$ such that $F(x, G(x, z))=z$. Then $\mathrm{T}$ has a local smooth inverse $\mathrm{id}_X\times G$. 
\end{proof}
 \subsection{Odd Euler class}\label{subsubsec:odd} We now show Theorem \ref{thm:nah} for odd $k$. Let $p:\Sigma_{2g-1}\to\Sigma_g$ be an arbitrary unramified double cover. Given $(X,D, \Phi)\in\mathcal{E}(g, k)$, by lifting everything via $p$, we get an element $(\hat{X}, \hat{D}, \hat{\Phi})\in \mathcal{E}(2g-1, 2k)$. The even Euler class case from \S\ref{subsubsec:even} then defines a map 
 \begin{align*}
    \hat{\mathrm{T}}:\mathcal{E}(g,k)&\longrightarrow \Teich_{2g-1}\times\chi_{g,2k}\\
    (X, D, \Phi)&\longrightarrow \mathrm{T}(\hat{X}, \hat{D}, \hat{\Phi})=(\hat{X}, \hat{\rho}).
 \end{align*}
 Let $(E,\phi)$ be as in \S\ref{subsubsec:even}. Then by \cite[\S 4]{Biswas2021}, the $\mathrm{PSL}(2,\mathbb{C})$-Higgs bundle $(\mathbb{P}(E), \phi)$ over $\hat{X}$ descends to a $\mathrm{PSL}(2,\mathbb{C})$-Higgs bundle on $X$. From the uniqueness of the solutions to the Hitchin equations, we see that the harmonic metric over $\hat{X}$ must be a lift of the harmonic metric on $X$. Therefore $\hat{\rho}:\pi_1(\Sigma_{2g-1})\to\mathrm{PSL}(2,\mathbb{R})$ can be uniquely extended to a representation $\rho:\pi_1(\Sigma_g)\to\mathrm{PSL}(2,\mathbb{R})$ such that $\rho\circ p_*=\tilde{\rho}$. It follows that $\hat{\mathrm{T}}$ descends to a smooth map $\mathcal{E}(g,k)\to\Teich_g\times\chi_{g,k}$. The fact that this map is a diffeomorphism follows as in Step 4 in \S\ref{subsubsec:even}.
\bibliographystyle{amsplain}
\bibliography{main}

\providecommand{\bysame}{\leavevmode\hbox to3em{\hrulefill}\thinspace}
\providecommand{\MR}{\relax\ifhmode\unskip\space\fi MR }
\providecommand{\MRhref}[2]{%
  \href{http://www.ams.org/mathscinet-getitem?mr=#1}{#2}
}
\providecommand{\href}[2]{#2}
\begin{thebibliography}{10}

\bibitem{Amoros2014}
J.~Amorós, \emph{{Fundamental Groups of Compact Kähler Manifolds}},
  Mathematical Surveys and Monographs, no. v.44, American Mathematical Society,
  Providence, 2014, Description based upon print version of record.

\bibitem{Andreotti1960}
Aldo Andreotti and Wilhelm Stoll, \emph{{Extension of Holomorphic Maps}},
  Annals of Mathematics \textbf{72} (1960), no.~2, 312.

\bibitem{aronszajn1956unique}
N.~Aronszajn, \emph{A unique continuation theorem for solutions of elliptic
  partial differential equations or inequalities of second order}, J. Math.
  Pures Appl. \textbf{36} (1957), no.~9, 235--249.

\bibitem{Bingener1983}
Jürgen Bingener, \emph{{On Deformations of Kähler Spaces. I.}}, Mathematische
  Zeitschrift \textbf{182} (1983), 505--536.

\bibitem{Biswas2021}
Indranil Biswas, Steven Bradlow, Sorin Dumitrescu, and Sebastian Heller,
  \emph{{Uniformization of branched surfaces and Higgs bundles}}, International
  Journal of Mathematics \textbf{32} (2021), no.~13.

\bibitem{Boggi2000}
M.~Boggi and M.~Pikaart, \emph{Galois covers of moduli of curves}, Compositio
  Mathematica \textbf{120} (2000), no.~2, 171--191.

\bibitem{Bridson2010}
Martin~R. Bridson, \emph{{Semisimple actions of mapping class groups on CAT(0)
  spaces}}, pp.~1--14, Cambridge University Press, February 2010.

\bibitem{Bryant1991}
Robert~L. Bryant, Shiing shen Chern, and Robert~B. Gardner (eds.),
  \emph{Exterior differential systems}, Mathematical Sciences Research
  Institute publications, no.~18, Springer, New York, 1991, Literaturverz. S.
  [462] - 470.

\bibitem{Carlson1989}
James~A. Carlson and Domingo Toledo, \emph{{Harmonic mappings of Kähler
  manifolds to locally symmetric spaces}}, Publications mathématiques de
  l’IHÉS \textbf{69} (1989), no.~1, 173--201.

\bibitem{Demailly2012}
Jean-Pierre Demailly, \emph{Complex analytic and differential geometry},
  OpenContent, 2012, available at
  \url{https://www-fourier.ujf-grenoble.fr/~demailly/manuscripts/agbook.pdf}.

\bibitem{Deroin2019}
Bertrand Deroin and Nicolas Tholozan, \emph{{Supra-Maximal Representations From
  Fundamental Groups Of Punctured Spheres To Psl(2, R)}}, {Annales
  Scientifiques de l'{\'E}cole Normale Sup{\'e}rieure} (2019).

\bibitem{Eells1981}
J~Eells and L~Lemaire, \emph{Deformations of metrics and associated harmonic
  maps}, Proceedings of the Indian Academy of Sciences - Section A \textbf{90}
  (1981), no.~1, 33--45.

\bibitem{Farb2011}
Benson Farb and Dan Margalit, \emph{A primer on mapping class groups},
  Princeton University Press, December 2011.

\bibitem{Fulton1981}
William Fulton and Robert Lazarsfeld, \emph{Connectivity and its applications
  in algebraic geometry}, pp.~26--92, Springer Berlin Heidelberg, 1981.

\bibitem{Grauert1984}
Hans Grauert and Reinhold Remmert, \emph{Coherent analytic sheaves}, Springer
  Berlin Heidelberg, 1984.

\bibitem{Gromov2012}
M.~Gromov, \emph{{Super stable K{\"a}hlerian horseshoe?}}, vol. 9783642288210,
  pp.~151--229, Springer-Verlag Berlin Heidelberg, June 2012 (English (US)).

\bibitem{Grauert1994}
R.~Remmert H.~Grauert, Th.~Peternell (ed.), \emph{{Several Complex Variables
  VII}}, Springer Berlin Heidelberg, 1994.

\bibitem{Hironaka1964}
Heisuke Hironaka, \emph{{Resolution of Singularities of an Algebraic Variety
  Over a Field of Characteristic Zero: I}}, The Annals of Mathematics
  \textbf{79} (1964), no.~1, 109.

\bibitem{Hironaka1964a}
\bysame, \emph{{Resolution of Singularities of an Algebraic Variety Over a
  Field of Characteristic Zero: II}}, The Annals of Mathematics \textbf{79}
  (1964), no.~2, 205.

\bibitem{Hitchin1987}
Nigel Hitchin, \emph{The self-duality equations on a {R}iemann surface},
  Proceedings of the London Mathematical Society \textbf{s3-55} (1987), no.~1,
  59--126.

\bibitem{Kelley2005}
John~L. Kelley, \emph{General topology}, repr. of the 1955 ed. by van nostrand,
  new york ed., Graduate texts in mathematics, no.~27, Springer, New York,
  2005.

\bibitem{Kollar2011}
János Kollár, \emph{{Lectures on resolution of singularities}}, Annals of
  mathematics studies, no. no. 166, Princeton University Press, Princeton,
  2011.

\bibitem{Li2019}
Qiongling Li, \emph{{An Introduction to Higgs Bundles via Harmonic Maps}},
  Symmetry, Integrability and Geometry: Methods and Applications (2019).

\bibitem{Loustau2020}
Brice Loustau, \emph{{Harmonic maps from Kähler manifolds}},  (2020).

\bibitem{Markovic2022}
Vladimir Marković, \emph{Unramified correspondences and virtual properties of
  mapping class groups}, Bulletin of the London Mathematical Society
  \textbf{54} (2022), no.~6, 2324--2337.

\bibitem{Markovic2024}
Vladimir Marković and Ognjen Tošić, \emph{{The second variation of the Hodge
  norm and higher Prym representations}}, Journal of Topology \textbf{17}
  (2024), no.~1.

\bibitem{Meier1986}
Michael Meier, \emph{Removable singularities of harmonic maps and an
  application to minimal submanifolds}, Indiana University Mathematics Journal
  \textbf{35} (1986), 705--726 (English).

\bibitem{Sagman2021}
Nathaniel Sagman, \emph{A factorization theorem for harmonic maps}, The Journal
  of Geometric Analysis \textbf{31} (2021), no.~12, 11714--11740.

\bibitem{Sampson1978}
J.~H. Sampson, \emph{Some properties and applications of harmonic mappings},
  Annales scientifiques de l'\'Ecole Normale Sup\'erieure \textbf{Ser. 4, 11}
  (1978), no.~2, 211--228 (en). \MR{80b:58031}

\bibitem{Sampson1986}
\bysame, \emph{{Applications of harmonic maps to Kähler geometry}}, 1986,
  pp.~125--134.

\bibitem{Schoen1997}
Richard~M. Schoen, \emph{Lectures on harmonic maps}, Conference proceedings and
  lecture notes in geometry and topology, no. v. 2, International Press,
  Cambridge, MA, 1997.

\bibitem{Siu1980}
Yum-Tong Siu, \emph{{The Complex-Analyticity of Harmonic Maps and the Strong
  Rigidity of Compact K\"ahler Manifolds}}, The Annals of Mathematics
  \textbf{112} (1980), no.~1, 73.

\bibitem{Slegers2021}
Ivo Slegers, \emph{Equivariant harmonic maps depend real analytically on the
  representation}, manuscripta mathematica \textbf{169} (2021), no.~3–4,
  633--648.

\bibitem{Sunada1979}
Toshikazu Sunada, \emph{Rigidity of certain harmonic mappings}, Inventiones
  Mathematicae \textbf{51} (1979), no.~3, 297--307.

\bibitem{Toledo2011}
Domingo Toledo, \emph{{Hermitian Curvature and Plurisubharmonicity of Energy on
  Teichm{\"u}ller Space}}, Geometric and Functional Analysis \textbf{22}
  (2011), 1015--1032.

\bibitem{Tosic2023}
Ognjen Tošić, \emph{{Non-Strict Plurisubharmonicity of Energy on Teichmüller
  Space}}, International Mathematics Research Notices (2024).

\bibitem{Varouchas1989}
Jean Varouchas, \emph{{K\"ahler spaces and proper open morphisms}},
  Mathematische Annalen \textbf{283} (1989), no.~1, 13--52.

\end{thebibliography}
\end{document}